\documentclass[12pt]{amsart}
\usepackage{amsmath,amscd,amssymb,amsfonts}
\usepackage[all]{xy}

\usepackage{geometry}
\newcommand{\bC}{{\mathbb C}}

\newcommand{\bH}{{\mathbb H}}

\newcommand{\bP}{{\mathbb P}}
\newcommand{\bR}{{\mathbb R}}
\newcommand{\bQ}{{\mathbb Q}}
\newcommand{\bZ}{{\mathbb Z}}

\newcommand{\cC}{{\mathcal C}}
\newcommand{\cD}{{\mathcal D}}

\newcommand{\cH}{{\mathcal H}}
\newcommand{\cI}{{\mathcal I}}
\newcommand{\cL}{{\mathcal L}}

\newcommand{\cZ}{{\mathcal Z}}
\newcommand{\cP}{{\mathcal P}}
\newcommand{\cS}{{\mathcal S}}

\newcommand{\cX}{{\mathcal X}}

\newcommand{\rk}{{\rm rk}}
\newcommand{\Hom}{{\rm{Hom}}}

\newcommand{\im}{{\rm{Im}}}

\newcommand{\Cok}{{\rm{Coker}}}
\newcommand{\Ker}{{\rm{Ker}}}
\newcommand{\wti}{\widetilde}

\newcommand{\ra}{\rightarrow}
\newcommand{\lra}{\longrightarrow}

\def\coker{{\rm {coker}}}
\def\im{{\rm {Image}}}

\def\var{{\it {var}}}
\def\Var{{\it {Var}}}
\def\can{{\it {can}}}
\def\sp{{\it {sp}}}
\def\rat{{\it {rat}}}

\newcommand{\MHM}{{MHM}}

\theoremstyle{plain}
\newtheorem{thm}{Theorem}[section]
\newtheorem{cor}[thm]{Corollary}

\newtheorem{lem}[thm]{Lemma}
\newtheorem{prop}[thm]{Proposition}

\newtheorem{ass}[thm]{Assumption}

\theoremstyle{definition}
\newtheorem{df}[thm]{Definition}
\newtheorem{rem}[thm]{Remark}
\newtheorem{example}[thm]{Example}

\def\be{\begin{equation}}
\def\ee{\end{equation}}
\def\bt{\begin{thm}}
\def\et{\end{thm}}
\def\bc{\begin{cor}}
\def\ec{\end{cor}}
\def\br{\begin{rem}}
\def\er{\end{rem}}
\def\bp{\begin{prop}}
\def\ep{\end{prop}}
\def\bl{\begin{lem}}
\def\el{\end{lem}}
\def\bn{\begin{enumerate}}
\def\en{\end{enumerate}}
\def\bex{\begin{example}}
\def\eex{\end{example}}
\def\bd{\begin{df}}
\def\ed{\end{df}}

\title[Intersection spaces and perverse sheaves]{Intersection spaces, perverse sheaves \\ and type IIB string theory}
\author{Markus Banagl}
\address{Mathematisches Institut, Universit\"at Heidelberg, Im Neuenheimer Feld 288, 69120 Heidelberg, Germany}
\email{banagl@mathi.uni-heidelberg.de}
\author{Nero Budur}
\address{Department of Mathematics,
University of Notre Dame, 255 Hurley Hall, IN 46556, USA} \email{nbudur@nd.edu}
\author{Lauren\c{t}iu Maxim}
\address{Department of Mathematics, University of Wisconsin, 480 Lincoln Drive, Madison, WI 53706, USA}
\email {maxim@math.wisc.edu}



\thanks {The first author was in part supported by a research grant of the
 Deutsche Forschungsgemeinschaft. The second author was supported by the Simons Foundation grant \#245850. The third author was supported by the NSF-1005338.}

\begin{document}

\begin{abstract}
The method of intersection spaces associates rational Poincar\'e complexes to
singular stratified spaces. For a conifold transition, the resulting cohomology theory
yields the correct count of all present massless 3-branes in type IIB string theory,
while intersection cohomology yields the correct count of massless 2-branes in type IIA theory. 
For complex projective hypersurfaces with an isolated singularity, we show that
the cohomology of intersection spaces is the hypercohomology of a perverse sheaf,
the intersection space complex, 
on the hypersurface.  Moreover,  the intersection space complex underlies a mixed Hodge module, so its hypercohomology groups carry canonical mixed Hodge structures.
For a large class of singularities, e.g., weighted homogeneous ones, global Poincar\'e duality is induced by a more refined
Verdier self-duality isomorphism for this perverse sheaf.
For such singularities, we prove furthermore 
that the pushforward of the constant sheaf of a nearby smooth deformation
under the specialization map to the singular space splits off the intersection
space complex as a direct summand.
The complementary summand is the contribution of the singularity.
Thus, we obtain for such hypersurfaces a mirror statement of the 
Beilinson-Bernstein-Deligne decomposition 
of the pushforward of the constant sheaf under an algebraic
resolution map into the intersection sheaf plus
contributions from the singularities. 
\end{abstract}

\maketitle

\tableofcontents

\section{Introduction}\label{intro}

In addition to the four dimensions that model our space-time, string theory requires six dimensions for a string to vibrate. Supersymmetry considerations force these six real dimensions to be a Calabi-Yau space. However, given the multitude of known topologically distinct Calabi-Yau $3$-folds, the string model remains undetermined. So it is  important to have mechanisms that allow one to move from one Calabi-Yau space to another. In Physics, a solution to this problem was first proposed by Green-H\"ubsch \cite{GH1, GH2}  who conjectured that topologically distinct Calabi-Yau's are connected to each other by means of {\it conifold transitions}, which  induce a phase transition between the corresponding string models. 

A conifold transition starts out with a smooth Calabi-Yau $3$-fold, passes through a singular variety --- 
{\it the conifold} --- by a deformation of complex structure, and arrives at a topologically distinct smooth Calabi-Yau $3$-fold by a small resolution of singularities. The deformation collapses embedded three-spheres (the {\it vanishing cycles}) to isolated ordinary double points, while the resolution resolves the singular points by replacing them with  $\bC\bP^1$'s.   In Physics, the topological change was interpreted by Strominger by the condensation of massive black holes to massless ones. In type IIA string theory, there are charged two-branes that wrap around the $\bC\bP^1$ $2$-cycles, and which become massless when these $2$-cycles are collapsed to points by the (small) resolution map. Goresky-MacPherson's intersection homology \cite{GM1}, \cite{GM2}
of the conifold accounts for all of these massless two-branes
(\cite[Proposition 3.8]{Ba}), and since it also satisfies
Poincar\'e duality, may be viewed as
a physically correct homology theory for type IIA string theory. Similarly, in type IIB string theory there are charged three-branes wrapped around the vanishing cycles, and which become massless as these vanishing cycles are collapsed by the deformation of complex structure. Neither ordinary homology nor intersection homology of the conifold account for these massless three-branes. So  a natural problem is to find a physically correct homology theory for the IIB string theory. A solution to this question was suggested by the first author in \cite{Ba} via his {\it intersection space homology theory}.

In \cite{Ba}, the first author develops a homotopy-theoretic method which associates to 
certain types of singular spaces $X$ (e.g., a conifold) a CW complex $IX$, called the {\it intersection space} of $X$, which is a (reduced) rational Poincar\'e complex, i.e., its reduced homology groups satisfy Poincar\'e Duality over the rationals.
The intersection space $IX$ associated to a singular space $X$ is constructed by replacing links of singularities of $X$ by their corresponding Moore approximations, a process called {\it spatial homology truncation}.  The {\it intersection space homology} $$HI_*(X;\bQ):={H}_*(IX;\bQ)$$  is not isomorphic to the intersection homology of the space $X$, and in fact it can be seen that in the middle degree and for isolated singularities, 
this new theory takes more cycles into account than intersection homology. 
For a conifold $X$, Proposition 3.6 and
Theorem 3.9 in \cite{Ba} establish that the dimension of $HI_3 (X)$ equals the number
of physically present massless 3-branes in IIB theory. \\ 

In  mirror symmetry, given a Calabi-Yau $3$-fold $X$, the {\it mirror map} associates to it another Calabi-Yau $3$-fold $Y$ so that type IIB string theory on $\bR^4 \times X$ corresponds to type IIA string theory on $\bR^4 \times Y$. If $X$ and $Y$ are smooth, their Betti numbers are related by precise algebraic identities, e.g., $\beta_3(Y)=\beta_2(X)+\beta_4(X)+2$, etc.  
Morrison \cite{Mor} conjectured that the mirror of a conifold transition is again a conifold transition, but performed in the reverse order (i.e.,  by exchanging resolutions and deformations). Thus, if  $X$ and $Y$ are mirrored conifolds (in mirrored conifold transitions), the intersection space homology of one space and the intersection homology of the mirror space form a {\it mirror-pair}, in the sense that 
$ \beta_3(IY)=I\beta_2(X)+I\beta_4(X)+2,$ etc., where $I\beta_i$ denotes the $i$-th intersection homology Betti number (see \cite{Ba} for details).
This suggests that it should be possible to compute the intersection space homology $HI_*(X;\bQ)$ of a variety $X$ in terms of the topology of a smoothing deformation, by ``mirroring" known results relating the intersection homology groups $IH_*(X;\bQ)$ of $X$ to the topology of a resolution of singularities. 
Moreover, the above identity of Betti numbers  can serve as a beacon in constructing a mirror $Y$ for a given singular $X$, as it restricts the topology of those $Y$ that can act as a mirror of $X$.

This point of view was exploited in \cite{BM}, where the first and third authors  considered the case of a  hypersurface $X \subset \bC\bP^{n+1}$ with only isolated singularities. For simplicity, let us assume that  $X$ has only one isolated singular point $x$, with Milnor fiber $F_x$ and local monodromy operator $T_x:H_n(F_x) \to H_n(F_x)$. Let $X_s$ be a nearby smoothing of $X$. Then it is shown in \cite{BM} that $H_*(IX;\bQ)$ is a vector subspace of $H_*(X_s;\bQ)$, while an isomorphism holds iff the local monodromy operator $T_x$ is trivial (i.e., in the case when $X$ is a {\it small degeneration} of $X_s$). This result can be viewed as mirroring the well-known fact that the intersection homology groups $IH_i(X;\bQ)$ of $X$ are vector subspaces of the corresponding homology groups $H_i(\widetilde{X};\bQ)$ of any resolution $\widetilde{X}$ of $X$, with an isomorphism in the case of a small resolution (e.g., see \cite{DM,GM3}).   

Guided by a similar  philosophy, in this paper we construct a perverse sheaf $\cI\cS_X$, the {\it intersection-space complex}, whose global hypercohomology calculates (abstractly) the intersection space cohomology groups of a projective hypersurface $X\subset \bC\bP^{n+1}$ with one isolated singular point. Our result ``mirrors" the fact that the intersection cohomology groups can be computed from a perverse sheaf, namely the intersection cohomology complex $\cI\cC_X$. We would like to point out that 
for general $X$ there cannot exist a perverse sheaf $\mathcal{P}$ on $X$
such that $HI^\ast (X;\bQ)$ can be computed from the hypercohomology
group $\bH^\ast (X;\mathcal{P})$, as follows from the stalk vanishing
conditions that such a $\mathcal{P}$ satisfies. However, 
Theorem \ref{thmIS} of the present paper shows that this goal \emph{can} be achieved in the case when $X$ is a hypersurface with only isolated singularities in a (smooth) deformation space, this being in fact the main source of examples for conifold transitions. Furthermore, by construction, 
 the intersection space complex $\cI\cS_X$ underlies a mixed Hodge module, therefore its hypercohomology groups carry canonical mixed Hodge structures. This result ``mirrors" the corresponding one for the intersection cohomology complex $\cI\cC_X$.

It follows from the above interpretation of intersection space cohomology that the groups $\bH^\ast (X;\cI \cS_X)$ satisfy Poincar\'e duality globally,
which immediately raises the question, whether this duality is induced by a more
powerful (Verdier-) self-duality isomorphism $\cD (\cI \cS_X)\simeq \cI\cS_X$ in the
derived category of constructible bounded sheaf complexes on $X$. 
Part $({c})$ of our main Theorem \ref{thmIS}  (see also Corollary \ref{corselfdual}) affirms that this is indeed the case, provided the local monodromy $T_x$ at the singular point is semi-simple in the eigenvalue $1$. This assumption is satisfied by a large class of isolated singularities, e.g., the weighted homogeneous ones.

Our fourth result ``mirrors'' the Beilinson-Bernstein-Deligne decomposition 
\cite{BBD} of the pushforward $Rf_*\bQ_{\widetilde{X}}[n]$ of the constant sheaf $\bQ_{\widetilde{X}}$ under an algebraic
resolution map $f:\widetilde{X} \to X$ into the intersection sheaf $\cI\cC_X$ of $X$ plus
contributions from the singularities of $X$. 
Suppose that $X$ sits as $X=\pi^{-1}(0)$ in a family $\pi: \wti{X}\to S$ of
projective hypersurfaces over a small disc around $0 \in \bC$ such that $X_s = \pi^{-1} (s)$ is smooth over nearby
$s\in S,$ $s\not= 0$. Under the above assumption on 
the local monodromy $T_x$ at the singular point,
we prove (cf. Theorem \ref{thmIS}({c}) and Corollary \ref{corsplitting}) that the nearby cycle complex 
$\psi_\pi\bQ_{\wti{X}}[n]$, a perverse sheaf on $X$, splits off
the intersection space complex $\cI\cS_X$ as a direct summand (in the category of
perverse sheaves). The complementary summand has the interpretation as being
contributed by the singularity $x$, since it is supported only over $\{ x \}$.
For $s$ sufficiently close to $0$, there is a map 
$sp: X_s \to X$, the specialization map. It follows by the decomposition of the nearby cycle
complex that the (derived) pushforward $Rsp_* \bQ_{X_s}[n]$ of the constant
sheaf on a nearby smoothing of $X$ splits off $\cI\cS_X$ as a direct summand,
\[ Rsp_* \bQ_{X_s}[n] \simeq \cI\cS_X \oplus \cC. \]

Finally, we would like to point out that since our paper is about a certain perverse sheaf
and its properties, and perverse sheaves are somewhat complicated objects,
it would be valuable to have an alternative, more elementary description of the perverse sheaf under consideration. To this end, we note that there are various more ``elementary" descriptions of the category of perverse sheaves available, for example the zig-zag category of MacPherson-Vilonen 
\cite{MV}. The latter description is particularly applicable for stratifications
whose singular strata are contractible. Since this is the case in the present
paper, it is desirable to understand the intersection space complex $\cI\cS_X$ and
its properties and associated short exact sequence
also on the level of zig-zags. We do provide such an analysis in Sections \ref{split} and \ref{sd}. 
In particular, under the above technical assumption on the local monodromy operator, we derive the splitting and the self-duality of the intersection space complex also directly on the level of zig-zags.\\

In \cite{huebschworkingdef}, T. H\"ubsch asks for a homology theory
$SH_\ast$ (``stringy homology'')
on $3$-folds $X,$ whose singular set $\Sigma$ contains only isolated singularities, such that \\

\noindent (SH1) $SH_\ast$ satisfies Poincar\'e duality, \\
\noindent (SH2) $SH_i (X) \cong H_i (X-\Sigma)$ for $i<3$, \\
\noindent (SH3) $SH_3 (X)$ is an extension of $H_3 (X)$ by
$\ker (H_3 (X-\Sigma)\rightarrow H_3 (X)),$ \\
\noindent (SH4) $SH_i (X) \cong H_i (X)$ for $i>3$.\\

\noindent Such a theory would record both the type IIA \emph{and} the type IIB
massless D-branes simultaneously.
Intersection homology satisfies all of these axioms with the exception
of axiom (SH3).  
Regarding (SH3), H\"ubsch notes further that
``the precise nature of this extension is to be determined
from the as yet unspecified general cohomology theory.'' Using the
homology of intersection spaces one
obtains an answer: The group
$HI_3 (X;\bQ)$ satisfies axiom (SH3) for any $3$-fold $X$ with isolated
singularities and simply connected links.
On the other hand, $HI_\ast (X;\bQ)$
does not satisfy axiom (SH2) (and thus, by Poincar\'e duality, does not
satisfy (SH4)), although it does satisfy (SH1) (in addition to (SH3)).
The pair $(IH_\ast (X;\bQ), HI_\ast (X;\bQ))$ does contain all the 
information that $SH_\ast (X)$ satisfying
(SH1)--(SH4) would contain and so may be regarded as a solution to
H\"ubsch' problem. In fact, one could set
\[ SH_i (X) = \begin{cases} IH_i (X;\bQ),& i\not= 3, \\ HI_i (X), & i=3.
  \end{cases} \]
This $SH_\ast$ then satisfies all axioms (SH1)--(SH4). 
A construction of $SH_\ast$ using the 
description of perverse sheaves by MacPherson-Vilonen's  zig-zags
\cite{MV} has been given by 
A. Rahman in \cite{R} for isolated singularities.
As already mentioned  above, zig-zags are also used in the present paper to obtain topological interpretations of our splitting and self-duality results.\\

\textbf{Acknowledgments.} 
We thank Morihiko Saito for useful discussions. We also thank J\"org Sch\"urmann for valuable comments on an earlier version of this paper, and for sharing with us his preprint \cite{DMSS}.

\section{Prerequisites}\label{prereq}

\subsection{Isolated hypersurface singularities}\label{hyper}
Let $f$ be a homogeneous polynomial in $n+2$ variables with complex coefficients such that the
complex projective hypersurface
\[ X = X(f) = \{ x\in \mathbb{P}^{n+1} ~|~ f(x)=0 \} \]
has only one isolated singularity $x$. 
Locally, identifying $x$ with the origin of $\bC^{n+1}$, the singularity is described
by a reduced analytic function germ
\[ g: (\bC^{n+1},0)\longrightarrow (\bC,0). \]
Let $B_\epsilon \subset \bC^{n+1}$ be a closed ball of radius $\epsilon >0$ centered at the origin and let
$S_\epsilon$ be its boundary, a sphere of dimension $2n+1$. 
Then, according to Milnor \cite{Mi}, for $\epsilon$ small enough,  
the intersection $X\cap B_\epsilon$ is homeomorphic to the cone over the {\it link}
$L = X\cap S_\epsilon = \{ g=0 \} \cap S_\epsilon$ of the singularity $x$, and the {\it Milnor map} of $g$ at radius $\epsilon,$
\[ \frac{g}{|g|}: S_\epsilon \setminus L \longrightarrow S^1, \]
is a (locally trivial) fibration. 
The link $L$ is a $(n-2)$-connected $(2n-1)$-dimensional submanifold
of $S_\epsilon$.
The fiber $F^{\circ}$ of the Milnor map is an open smooth manifold of real dimension
$2n$, which shall be called the {\it open Milnor fiber} at $x$. Let $F$ be the closure in $S_{\epsilon}$ of the fiber of $g/|g|$ over $1\in S^1$.
Then $F$, the {\it closed Milnor fiber} of the singularity, is a compact manifold with 
boundary $\partial F = L,$ the link of $x$. Note that $F^{\circ}$ and $F$ are homotopy equivalent, and in fact they have the homotopy type of a bouquet of $n$-spheres, see \cite{Mi}. The number $\mu$ of spheres in
this bouquet is called the {\it Milnor number} and can be computed as
\[ \mu = \dim_\bC \frac{\mathcal{O}_{n+1}}{J_g}, \]
with $\mathcal{O}_{n+1} = \bC \{ x_0, \ldots, x_n \}$ the $\bC$-algebra of all convergent
power series in $x_0, \ldots, x_n$, and $J_g = (\partial g/\partial x_0, \ldots,
\partial g/\partial x_n)$ the Jacobian ideal of the singularity. 
Associated with the Milnor fibration $F^{\circ} \hookrightarrow S_\epsilon - L \to S^1$ is a 
monodromy homeomorphism $h: F^{\circ} \to F^{\circ}$. Using the identity
$L \to L,$ $h$ extends to a homeomorphism
$h: F  \to F$ because $L$ is the binding of the corresponding open
book decomposition of $S_\epsilon$. This homeomorphism induces the (local) {\it monodromy operator}
\[ T_x = h_\ast: H_\ast (F;\bQ) \stackrel{\cong}{\longrightarrow} H_\ast (F;\bQ). \]
If $n \geq 2$, the difference between the monodromy operator and the identity fits into the
{\it Wang sequence} of the fibration,
\begin{equation} \label{wang}
0 \to H_{n+1} (S_\epsilon - L;\bQ) \to H_n (F;\bQ) 
 \stackrel{T_x-1}{\longrightarrow} H_n (F;\bQ) \to
 H_n (S_\epsilon - L;\bQ)\to 0.
\end{equation}

\subsection{Intersection space (co)homology of projective hypersurfaces}\label{IShyp}
Let $X$ be a complex projective  hypersurface of dimension $n> 2$ with only one isolated singular point $x$. The assumption on dimension is needed to assure that the link $L$ of $x$ is simply-connected, so the intersection space $IX$ can be defined as in \cite{Ba}. The actual definition of an intersection space is not needed in this
paper, only the calculation of Betti numbers, as described in the next theorem, will be used in the sequel. 
Nevertheless, let us indicate briefly how $IX$ is obtained from $X$. Let $M$ be the complement
of an open cone neighborhood of $x$ so that $M$ is a compact manifold with boundary
$\partial M = L$. Given an integer $k$, a \emph{spatial homology $k$-truncation} is a topological space
$L_{<k}$ such that $H_i (L_{<k})=0$ for $i\geq k$,
together with a continuous map $f: L_{<k} \to L$ which induces a homology-isomorphism
in degrees $i<k$. Using the truncation value $k=n,$ the intersection space $IX$ is the homotopy cofiber of the composition
\[ L_{<n} \stackrel{f}{\longrightarrow} L=\partial M \stackrel{\operatorname{incl}}{\longrightarrow} M. \]
Let $X_s$ be a nearby smooth deformation of $X$. Denote by $T_x$ the monodromy operator  on the middle cohomology of the Milnor fiber of the hypersurface germ $(X,x)$. Denote by $$HI^*(X;\bQ):=H^*(IX;\bQ)$$ the intersection-space cohomology of $X$.

\bt \label{thmBM}{\rm (\cite{BM}[Thm.4.1, Thm.5.2])}  Under the above assumptions and notations the following holds:
\[
 \dim HI^i(X;\bQ) =\left \{
\begin{array}{ll}
\dim H^i(X_s;\bQ)  & \text{ if } i\ne  n, 2n; \\
 \dim H^i(X_s;\bQ) - \rk (T_x- 1) & \text{ if } i=n;\\
0 & \text{ if }i=2n. \\
\end{array}\right.
\]
Moreover, if  $H_{n-1}(L;\bZ)$ is torsion-free, then for $i\ne n$ the above equality is given by canonical isomorphisms of vector spaces
induced by a continuous map.
\et

\subsection{Perverse sheaves}\label{perv} Let $X$ be a complex algebraic variety of complex dimension $n$, and $D^b(X)$ the bounded derived category of complexes of sheaves of rational vector spaces on $X$. If $\cX$ is a Whitney stratification of $X$, we say that $K \in D^b(X)$ is $\cX$-{\it cohomologically constructible} if, for all $i \in \bZ$ and any (pure) stratum $S$ of $\cX$, the cohomology sheaves $\cH^i(K)|_{S}$ are locally constant with finite dimensional stalks on $S$. We denote by $D^b_c(X)$ the derived category of bounded constructible sheaf complexes on $X$, i.e., the full subcategory of $D^b(X)$ consisting of those complexes which are cohomologically constructible with respect to some stratification of $X$. 

The abelian category of {\it perverse sheaves} on $X$ is the full sub-category $Perv(X)$ of $D^b_c(X)$ whose objects are characterized as follows. Assume $K \in D^b_c(X)$ is cohomologically constructible with respect to a stratification $\cX$ of $X$, and denote by $i_l:S_l \hookrightarrow X$ the corresponding embedding of a stratum of complex dimension $l$. Then $K$ is perverse if it satisfies  the following properties: 
\bn 
\item[(i)] {\it condition of support}:
\[ \cH^j(i_l^*K)=0, \ {\rm for \ any} \ l \ {\rm and} \  j \ {\rm with} \ j > -l, \]
\item[(ii)] {\it condition of cosupport}:
\[ \cH^j(i_l^!K)=0, \ {\rm for \ any} \ l \ {\rm and} \  j \ {\rm with} \ j < -l .\]
\en
If $X$ is smooth, then any $K \in D^b_c(X)$, which is constructible with respect to the intrinsic Whitney stratification of $X$ with only one stratum, is perverse if and only if $K$ is (up to a shift) just a local system. More precisely, in this case we have that $K \simeq \cH^{-n}(K)[n]$. More generally, if $K \in Perv(X)$ is supported on a closed $l$-dimensional stratum $S_l$, then $K \simeq \cH^{-l}(K)[l]$. 

Let us also recall here that the Verdier duality functor as well as restriction to open subsets preserve perverse objects.

In this paper, we will be  interested in the situation when the variety $X$ has only isolated singularities. For example, if $X$ has only one singular point $x$, then $X$ can be given a Whitney stratification $\mathcal{X}$ with only two strata: $\{x\}$ and $X \setminus \{x\}$. 
Denote by $i:\{x\} \hookrightarrow X$ and $j:X \setminus \{x\} \hookrightarrow X$ the corresponding closed and open embeddings. Then a complex $K \in D^b_c(X)$, which is constructible with respect to $\mathcal{X}$, is perverse on $X$ if $j^*K[-n]$ is cohomologically
a local system on $X^{\circ}:=X \setminus \{x\}$ and, moreover, the following two conditions hold:
 \[ H^j(i^*K)=0, \ {\rm for \ any} \ j > 0, \]
 and
 \[ H^j(i^!K)=0, \ {\rm for \ any} \ j < 0. \]

\subsection{Nearby and vanishing cycles}\label{nv}
Let $\pi:\wti{X}\ra S$ be a projective morphism from an $(n+1)$-dimensional complex manifold  onto a small disc $S$ around the origin in $\bC$. Assume $\pi$ to be smooth except over $0$. Denote by $X=\pi^{-1}(0)$ the singular zero-fiber and by $X_s=\pi^{-1}(s)$  ($s\ne 0$) the smooth projective variety which is the generic fiber of $\pi$. 

Let  $\psi_\pi, \phi_\pi: D^b_c (\wti{X}) \to D^b_c ({X})$ denote the nearby and, respectively, vanishing cycle functors of $\pi$ (e.g., see \cite{Di} and the references therein). These functors come equipped with monodromy automorphisms, both of which will be denoted by $T$. 
Instead of defining the nearby and vanishing cycle functors, we list their main properties as needed in this paper.
First, we have that 
\be\label{one}
H^i(X_s;\bQ)= \bH^i(X;\psi_\pi \bQ_{\wti{X}}),
\ee
and, for a point inclusion $i_x:\{x\}\hookrightarrow X$ with $F_x$ the Milnor fiber of the hypersurface singularity germ $(X,x)$,
\be
H^i(F_x;\bQ)= H^i(i_x^*\psi_\pi \bQ_{\wti{X}}) \ee
and
\be\label{456} \wti{H}^i(F_x;\bQ)= H^i(i_x^*\phi_\pi \bQ_{\wti{X}}),
\ee
with compatible monodromies $T_x$ and $T$. Moreover, the support of the vanishing cycles is 
\begin{equation}\label{eqSupp}
Supp (\phi_\pi \bQ_{\wti{X}}) = Sing (X),
\end{equation}
the singular locus of $X$, see \cite{Di}[Cor.6.1.18]. In particular, if $X$ has only isolated singularities, then:
\be
\bH^i(X;\phi_\pi \bQ_{\wti{X}}) = \bigoplus_{x \in Sing(X)} \wti{H}^i(F_x;\bQ).
\ee

By the definition of vanishing cycles, for every constructible sheaf complex $K\in D^b_c(\wti{X})$ there is a unique distinguished triangle
\begin{equation}\label{eqK}
t^*K \overset{\sp}{\lra} \psi_\pi K \overset{\can}{\lra} \phi_\pi K\mathop{\lra}^{[+1]}
\end{equation}
in $D^b_c(X)$, where  $t:X\hookrightarrow \wti{X}$ is the inclusion of the zero-fiber of $\pi$.  There is a similar distinguished triangle associated to the {\it variation morphism}, see \cite{KS}[p.351-352], namely:
\begin{equation}\label{varnew}
\phi_\pi K \mathop{\lra}^{\var} \psi_\pi K \lra t^![2] K \mathop{\lra}^{[+1]}.
\end{equation}
The variation morphism $\var:\phi_\pi K \to \psi_\pi K$  is heuristically defined by the cone of the pair of morphisms (but see the above reference \cite{KS} for a formal definition):
$$(0, T-1): [t^*K\to \psi_\pi K] \lra [0 \to \psi_\pi K].$$ Moreover, we have: $\can \circ \var=T-1$ and $\var \circ \can=T-1$.

The monodromy automorphism $T$ acting on the nearby and vanishing cycle functors has a Jordan decomposition $T=T_u \circ T_s=T_s \circ T_u$, where $T_s$ is semisimple (and locally of finite order) and  $T_u$ is  unipotent.  For any $\lambda \in \bQ$  and $K\in D^b_c(\wti{X})$, let us denote by  $\psi_{\pi,\lambda}K$ the generalized $\lambda$-eigenspace for $T$,
and similarly for $\phi_{\pi,\lambda}K$. By the definition of vanishing cycles, the canonical morphism $\can$ induces morphisms $$\can: \psi_{\pi,\lambda}K \lra \phi_{\pi,\lambda}K$$ which are isomorphisms for $\lambda \neq 1$, and there is a distinguished triangle 
\begin{equation}\label{eqK1}
t^*K \overset{\sp}{\lra} \psi_{\pi,1} K \overset{\can}{\lra} \phi_{\pi,1} K\mathop{\lra}^{[+1]} \ .
\end{equation}
There are decompositions
\be\label{678} \psi_{\pi}= \psi_{\pi,1} \oplus  \psi_{\pi,\neq1} \ \ {\rm and} \ \ \phi_{\pi}= \phi_{\pi,1} \oplus  \phi_{\pi,\neq1} \ee so that $T_s=1$ on $\psi_{\pi,1}$ and $\phi_{\pi,1}$, and $T_s$ has no $1$-eigenspace on $\psi_{\pi,\neq1}$ and $\phi_{\pi,\neq1}$. Moreover, $\can:\psi_{\pi,\neq1} \to \phi_{\pi,\neq1}$ and $\var: \phi_{\pi,\neq1} \to \psi_{\pi,\neq1}$ are isomorphisms.

Let $$N:=\log (T_u),$$ and define the morphism $\Var$
\begin{equation}
\phi_{\pi} K \mathop{\lra}^{\Var} \psi_{\pi} K
\end{equation}
by the cone of the pair $(0,N)$, see \cite{Sa}. (We omit the Tate twist $(-1)$ appearing in loc. cit., since it is not relevant for the topological considerations below.) Moreover, we have $\can \circ \Var=N$ and $\Var \circ \can=N$, and there is a distinguished triangle:
\begin{equation}\label{varnew1}
\phi_{\pi,1} K \mathop{\lra}^{\Var} \psi_{\pi,1} K \lra t^![2] K \mathop{\lra}^{[+1]}.
\end{equation}

The functors  $\psi_\pi[-1]$ and $\phi_\pi[-1]$ from $D^b_c(\wti{X})$ to $D^b_c(X)$ commute with the Verdier duality functor $\cD$ up to natural isomorphisms \cite{Ma}, and send perverse sheaves to perverse sheaves.  To simplify the notation, we denote the {\it perverse} nearby and vanishing cycle functors by  $${^p}\psi_\pi:=\psi_\pi[-1] \ \ \ {\rm and}  \ \ \ {^p}\phi_\pi:=\phi_\pi[-1],$$
respectively.
 
The functors ${^p}\psi_\pi$ and ${^p}\phi_\pi$ acting on perverse sheaves (such as $\bQ_{\wti{X}}[n+1]$) lift to functors $\psi^H_\pi:\MHM(\wti{X}) \to \MHM(X)$ and resp. $\phi^H_\pi: \MHM(\wti{X}) \to \MHM(X)$ defined on the category $\MHM(\wti{X})$ of {\it mixed Hodge modules} on $\wti{X}$, see \cite{Sa2}. More precisely, if $$\rat:\MHM(\wti{X}) \to Perv(\wti{X})$$ (and similarly for $X$) is the forgetful functor assigning to a mixed Hodge module the underlying perverse sheaf, then $$\rat \circ \psi^H_\pi={^p}\psi_\pi \circ \rat \ \ {\rm and} \ \ \rat \circ \phi^H_\pi={^p}\phi_\pi \circ \rat.$$ Moreover,  the above morphisms $\can$, $N$, $\Var$ and decompositions ${^p}\psi_\pi={^p}\psi_{\pi,1} \oplus {^p}\psi_{\pi,\neq 1}$ (and similarly for ${^p}\phi_\pi$) lift to the category of mixed Hodge modules, see \cite{Sa,Sa2} for details.
 
 The following semisimplicity criterion for perverse sheaves will be needed in Lemma \ref{lemnv} below, see \cite{Sa}[Lemma 5.1.4]  (as reformulated in \cite{Sa3}[(1.6)]):
 \begin{prop}\label{ss} Let $Z$ be a complex manifold and $K$ be a perverse sheaf on $Z$. Then  the following conditions are equivalent:
 \begin{enumerate}
 \item[(a)]  In the category $Perv(Z)$ one has a splitting
 $${^p}\phi_{g,1}(K)=\Ker\left(\Var:{^p}\phi_{g,1}(K) \to {^p}\psi_{g,1}(K) \right) \oplus \im \left(\can:{^p}\psi_{g,1}(K) \to {^p}\phi_{g,1}(K)  \right)$$
 for any locally defined holomorphic function $g$ on $Z$.
 \item[(b)] $K$ admits a decomposition by strict support, i.e., it can be written canonically as a direct sum of twisted intersection cohomology complexes.
 \end{enumerate}
 \end{prop}

We can now state one of the key technical results needed in the proof of our main theorem:
 \begin{lem}\label{lemnv} Let $\pi:\wti{X}\ra S$ be a projective morphism from an $(n+1)$-dimensional complex manifold  onto a small disc $S$ around the origin in $\bC$. Then, for any $i \in \bZ$, the restriction of the $\bQ$-vector space homomorphism $$\var_{\bQ}: \bH^i(X;\phi_\pi \bQ_{\wti{X}}) \to  \bH^i(X;\psi_\pi \bQ_{\wti{X}})$$ (induced by the variation morphism $\var:\phi_\pi \to \psi_\pi$) on the image of the endomorphism $T-1$ acting on $\bH^i(X;\phi_\pi \bQ_{\wti{X}})$ is one-to-one.
 \end{lem}
 
 \begin{proof} Let $\widehat{\pi}:X \to \{0\}$ be the restriction of $\pi$ to its zero-fiber, and denote by $s$ the coordinate function on the disc $S\subset \bC$ (with $s \circ \pi = \pi$). Then we have
 \begin{equation}
 \begin{split}
 \bH^i(X;\phi_\pi \bQ_{\wti{X}}) & \simeq \bH^{i-n}(X;{^p}\phi_\pi \bQ_{\wti{X}}[n+1]) \\ & \simeq 
 H^{i-n}(R\widehat{\pi}_*({^p}\phi_\pi\bQ_{\wti{X}}[n+1])) \\
& \simeq  {^p}\cH^{i-n}(R\widehat{\pi}_*({^p}\phi_\pi\bQ_{\wti{X}}[n+1])) \\ 
& \simeq  {^p}\cH^{i-n}(R\widehat{\pi}_*({^p}\phi_{s \circ \pi}\bQ_{\wti{X}}[n+1])) \\
& \simeq {^p}\phi_s( {^p}\cH^{i-n}(R\pi_*\bQ_{\wti{X}}[n+1]))
 \end{split}
 \end{equation}
 where ${^p}\cH$ denotes the perverse cohomology functor, and 
 the last identity follows by proper base-change (see \cite{Sa2}[Theorem 2.14]). Similarly, we have 
 $$\bH^i(X;\psi_\pi \bQ_{\wti{X}}) \simeq {^p}\psi_s( {^p}\cH^{i-n}(R\pi_*\bQ_{\wti{X}}[n+1])),$$ and these isomorphisms are compatible with the monodromy actions and they commute with the morphisms $\can$ and $\var$. 
 
Therefore, it suffices to prove the claim for the morphisms $\var$ and $T-1$ acting on $${^p}\phi_s( {^p}\cH^{j}(R\pi_*\bQ_{\wti{X}}[n+1])), \ \ j \in \bZ.$$
 Moreover, since $\var$ is an isomorphism on ${^p}\phi_{s,\neq 1}$, we can replace ${^p}\phi_s( {^p}\cH^{j}(R\pi_*\bQ_{\wti{X}}[n+1]))$ by ${^p}\phi_{s,1}( {^p}\cH^{j}(R\pi_*\bQ_{\wti{X}}[n+1]))$. But on ${^p}\phi_{s,1}( {^p}\cH^{j}(R\pi_*\bQ_{\wti{X}}[n+1]))$ we have that $T_s=1$  and $T=T_u$. 
Thus on this eigenspace, with $T_n$ the nilpotent morphism $T_n = 1-T_u,$
\[ N = \log T = \log T_u = -\sum_{k=1}^\infty \frac{1}{k}T_n^k =
  (T-1)(1+T_N), \]
where $T_N = \sum_{k=1}^\infty \frac{1}{k+1} T_n^k$ is nilpotent.
In particular, $1+T_N$ is an automorphism.
Hence the endomorphisms $N$ and $T-1$ acting on ${^p}\phi_{s,1}( {^p}\cH^{j}(R\pi_*\bQ_{\wti{X}}[n+1]))$ have the same image. Similarly, the morphisms $\var$ and $\Var$ also differ by an automorphism on ${^p}\phi_{s,1}( {^p}\cH^{j}(R\pi_*\bQ_{\wti{X}}[n+1]))$, so in particular they have the same kernel. Hence we can further replace $T-1$ by $N$, and $\var$ by $\Var$.
 
Next note that by the decomposition theorem of \cite{Sa}, the perverse sheaf ${^p}\cH^{j}(R\pi_*\bQ_{\wti{X}}[n+1])$ on $S$ admits a decomposition by strict support, i.e., it can be written canonically as a direct sum of twisted intersection cohomology complexes.
Therefore, by the semisimplicity criterion of Proposition \ref{ss}, applied to $S$ with the coordinate function $s$,  we have a splitting:
 $${^p}\phi_{s,1}( {^p}\cH^{j}(R\pi_*\bQ_{\wti{X}}[n+1]))=\Ker\left(\Var:{^p}\phi_{s,1} \to {^p}\psi_{s,1}\right) \oplus \im \left(\can:{^p}\psi_{s,1} \to {^p}\phi_{s,1}  \right).$$
 Moreover, since $\can \circ \Var=N$ on $\phi_{s,1}$, we have that $$\im(N) \subset \im \left(\can:{^p}\psi_{s,1} \to {^p}\phi_{s,1}  \right),$$ which by the above splitting is equivalent to 
 $$\im(N) \cap \Ker\left(\Var:{^p}\phi_{s,1} \to {^p}\psi_{s,1}\right) = \{ 0 \}.$$
 This finishes the proof of our claim.
 
  \end{proof}

\subsection{Zig-zags. Relation to perverse sheaves}\label{zz}
We will adapt the results of \cite{MV} to the situation we are interested in, namely, that of a $n$-dimensional complex algebraic variety with only one isolated singular point $x$. Let as above, denote by $X^{\circ}=X \setminus \{x\}$ the regular locus and $j :X^{\circ} \hookrightarrow X$ and $i : \{x\} \hookrightarrow X$
the open, respectively closed, embeddings. 
Throughout this paper, we will only consider complexes of sheaves (e.g., perverse sheaves) which are constructible with respect to the Whitney stratification of $X$ consisting of the two strata $X^{\circ}$ and $\{x\}$.

The {\it zig-zag category} $Z(X,x)$ is defined as follows. An object in $Z(X,x)$ consists of a tuple 
$(\cP,A,B,\alpha,\beta,\gamma)$, with $\cP \in Perv (X^{\circ})$
(hence $\cP=\cL[n]$, for $\cL$ a local system with finite dimensional stalks on $X^{\circ}$), and $A$ and $B$ fit into an exact sequence of vector spaces
\be
\begin{CD}
H^{-1}(i^*Rj_*\cP) @>\alpha>> A @>\beta>> B @>\gamma>> H^{0}(i^*Rj_*\cP). 
\end{CD}
\ee
A morphism in $Z(X,x)$ between two zig-zags $(\cP,A,B,\alpha,\beta,\gamma)$ and $(\cP',A',B',\alpha',\beta',\gamma')$ consists of a morphism $p:\cP \to \cP'$ in $Perv (X^{\circ})$, and vector space homomorphisms $A \to A'$, $B \to B'$, together with a commutative diagram:
\be
\begin{CD}
H^{-1}(i^*Rj_*\cP) @>\alpha>> A @>\beta>> B @>\gamma>> H^{0}(i^*Rj_*\cP)\\ 
@VV{p_*}V @VVV @VVV @VV{p_*}V \\
H^{-1}(i^*Rj_*\cP') @>\alpha'>> A' @>\beta'>> B' @>\gamma'>> H^{0}(i^*Rj_*\cP').
\end{CD}
\ee
The {\it zig-zag functor} $\cZ:Perv(X) \to Z(X,x)$ is defined by sending an object $K \in Perv(X)$ to the triple $(j^*K, H^0(i^!K), H^0(i^*K))$, together with the exact sequence:
\be\label{100}
\begin{CD}
H^{-1}(i^*Rj_*j^*K) @>>> H^0(i^!K) @>>> H^0(i^*K) @>>> H^{0}(i^*Rj_*j^*K). 
\end{CD}
\ee
A morphism $\kappa:K\to K'$ in $Perv(X)$ induces a morphism $\cZ (\kappa): \cZ(K)\to \cZ(K')$
given by applying the functors $H^* (i^* Rj_* j^* -),$ $H^0 (i^! -)$ and $H^0 (i^* -)$ to $\kappa$
to get the vertical maps of
\[ \begin{CD}
H^{-1}(i^*Rj_* j^* K) @>\alpha>> H^0(i^!K) @>\beta>> H^0(i^*K) @>\gamma>> H^{0}(i^*Rj_* j^* K)\\ 
@VVV @VVV @VVV @VVV \\
H^{-1}(i^*Rj_* j^* K') @>\alpha'>> H^0(i^!K') @>\beta'>> H^0(i^*K') @>\gamma'>> H^{0}(i^*Rj_* j^* K').
\end{CD} \]

\br\label{rem1}
The exact sequence (\ref{100}) is part of the cohomology  long exact sequence corresponding to the distinguished triangle \[ i^! K \to i^*K \to i^*Rj_*j^*K \overset{[1]}{\to} \ , \]  obtained by applying the functor $i^*$ to the attaching triangle: \[ i_!i^! K \to K \to Rj_*j^*K \overset{[1]}{\to} \ .\]
\er
Perverse sheaves and zig-zags are related by the following result of MacPherson and Vilonen:
\bt {\rm(\cite{MV}[Thm.2.1,Cor.2.2])} \label{MV}
\bn 
\item[(a)] The zig-zag functor $\cZ:Perv(X) \to Z(X,x)$ gives rise to a bijection between isomorphism classes of objects of $Perv(X)$ to the isomorphism classes of objects of $Z(X,x)$.
\item[(b)] For any two objects $K$ and $K'$ in $Perv(X)$, with $\beta$ and $\beta'$ denoting the maps $H^0(i^!K) \overset{\beta}{\to} H^0(i^*K)$ and $H^0(i^!K') \overset{\beta'}{\to} H^0(i^*K')$, 
there is an exact sequence:
\be\label{imp}
0 \to \Hom(\Cok(\beta), \Ker(\beta')) \to  \Hom(K,K') \overset{\cZ}{\to}  \Hom(\cZ K,\cZ K') \to 0.
\ee
In particular, if either $\beta$ or $\beta'$ is an isomorphism, then $\cZ$ induces an isomorphism
$$\Hom(K,K') \cong  \Hom(\cZ K,\cZ K').$$
\en
\et
 The following example is of interest to us:
\bex\label{zzn} ({\it Nearby cycles})\newline
Let us consider the situation described in Section \ref{nv}, i.e., a family of projective $n$-dimensional hypersurfaces $\pi:\wti{X} \to S$ with zero fiber $X:=\pi^{-1}(0)$ with only one isolated singularity $x$. Denote by $F$ the (closed) Milnor fiber, and by $L=\partial F$ the corresponding link at $x$. The zig-zag associated to the perverse sheaf $\psi_{\pi}(\bQ_{\wti{X}}[n]) \in Perv(X)$ consists of the triple $(\bQ_{X^{\circ}}[n],H^n(F,L;\bQ),H^n(F;\bQ))$, together with the exact sequence
\be\label{zzne}
H^{n-1}(L;\bQ) \to H^n(F,L;\bQ) \to H^n(F;\bQ) \to H^n(L;\bQ),
\ee
which is in fact the relevant part of the cohomology long exact sequence for the pair $(F,L)$.
Indeed, we have the following identifications: 
\bn
\item[] $H^k(i^*Rj_*j^*\psi_{\pi}(\bQ_{\wti{X}}[n])) \simeq H^k(i^*Rj_*\bQ_{X^{\circ}}[n]) \simeq H^{k+n}(i^*Rj_*\bQ_{X^{\circ}}) \simeq H^{k+n}(L;\bQ)$,
\item[] $H^0(i^*\psi_{\pi}(\bQ_{\wti{X}}[n])) \simeq H^n(i^*\psi_{\pi}\bQ_{\wti{X}}) \simeq H^n(F;\bQ)$,
\item[] $H^0(i^!\psi_{\pi}(\bQ_{\wti{X}}[n])) \simeq H^n_c(F^{\circ};\bQ) \simeq H^n(F,L;\bQ)$,
\en
with $F^{\circ}=F \setminus L$ denoting as before the open Milnor fiber. Finally, note that if $n \geq 2$, by using the fact that $F$ is $(n-1)$-connected we deduce that the leftmost arrow in (\ref{zzne}) is injective, while the rightmost arrow is surjective.
\eex


\section{Main results}\label{main}

In this section we define the intersection-space complex and study its properties.
\subsection{Construction}\label{cons}
Let $X$ be a complex projective  hypersurface of dimension $n> 2$ with only one isolated singular point $x$. Then the link $L$ of $x$ is simply-connected and the intersection space $IX$ is defined as in \cite{Ba}.
In this section, we construct a  perverse sheaf on $X$, which we call {\it the intersection-space complex} and denote it $\cI\cS_X$, such that \be\label{two} \dim \bH^i(X;\cI\cS_X[-n])=\dim HI^i(X;\bQ) \ee
for all $i$, except at $i=2n$.

As in Section \ref{nv}, let $\pi:\wti{X}\ra S$ be a deformation of $X$, where $S$ is a small disc around the origin in $\bC$, the total space $\wti{X}$ is smooth, and the fibers $X_s$ for $s\ne 0$ are smooth projective hypersurfaces in $\bP^{n+1}$. Define $\cC$ to be the image in the abelian category $Perv (X)$ of the morphism of perverse sheaves
\be
T-1 : \phi_\pi \bQ_{\wti{X}}[n]\lra \phi_\pi \bQ_{\wti{X}}[n], 
\ee
with $\phi_\pi$ the vanishing cycle functor for $\pi$. 
So we have a monomorphism of perverse sheaves
\be\label{mono}
\cC\hookrightarrow \phi_\pi \bQ_{\wti{X}}[n] .
\ee
By (\ref{eqSupp}), both perverse sheaves $\cC$ and $\phi_\pi \bQ_{\wti{X}}[n]$ are supported only on the singular point $x$.  Composing (\ref{mono}) with the {variation morphism} \be \var: \phi_\pi \bQ_{\wti{X}}[n]\lra \psi_\pi \bQ_{\wti{X}}[n],\ee we obtain a morphism of perverse sheaves
\begin{equation}\label{eqGamma}
\iota:  \cC \lra  \psi_\pi \bQ_{\wti{X}}[n] .
\end{equation}
Thus we can make the following definition:

\begin{df} The {\it intersection-space complex} of $X$ is defined as
\be
\cI\cS_X := \Cok (\iota: \cC \lra  \psi_\pi \bQ_{\wti{X}}[n])\ \in\ Perv(X).
\ee
\end{df}

It is clear from the definition that $\cI\cS_X$ is a perverse sheaf. In the next section, we show that $\cI\cS_X$ satisfies the identity (\ref{two}) on Betti numbers. The latter fact also motivates the terminology. Moreover, for certain types of singularities (e.g., weighted homogeneous), $\cI\cS_X$ is self-dual and it carries a decomposition similar to the celebrated BBD decomposition theorem \cite{BBD}.

\subsection{Main theorem}\label{mt}
The main result of this paper is the following:

\begin{thm}\label{thmIS} (a) The intersection-space complex $\cI\cS_X$ recovers the intersection-space cohomology. More precisely, there are abstract isomorphisms
\begin{align}
\bH^i(X;\cI\cS_X[-n])\simeq \left\{
\begin{array}{cc}
HI^i(X;\bQ) & \text{ if }i\ne 2n\\
H^{2n}(X_s;\bQ)=\bQ & \text{ if }i=2n.
\end{array}\right.
\end{align}

\noindent
(b) The hypercohomology groups $\bH^r(X;\cI\cS_X)$ carry natural mixed Hodge structures.

\bigskip

Moreover, if the local monodromy $T_x$ at $x$ is semi-simple in the eigenvalue $1$, then:

\noindent
(c) There is a canonical splitting 
\be
\psi_\pi\bQ_{\wti{X}}[n] \simeq \cI\cS_X \oplus \cC  .
\ee

\noindent
(d)  The intersection-space complex is self-dual. In particular, there is a non-degenerate pairing
\be
\bH^{-i}(X;\cI\cS_X) \times \bH^i(X;\cI\cS_X)\ra \bQ .
\ee

\end{thm}

Before proving the theorem, let us note the following:  
\begin{rem}\label{rem5}
\noindent{(i)} If $\pi$ is a {\it small deformation} of $X$, i.e., if the local monodromy operator $T_x$ is trivial, then $\cC \simeq 0$, so we get an isomorphism of perverse sheaves $\cI\cS_X \simeq \psi_\pi \bQ_{\wti{X}}[n]$. In view of the Betti identity (\ref{two}), this isomorphism can be interpreted as a sheaf-theoretical enhancement of the stability result from \cite{BM} mentioned in the introduction.

\noindent{(ii)} The above construction can be easily adapted to the situation of hypersurfaces with multiple isolated singular points. It then follows from $(a)$ and 
\cite{Ba}[Prop.3.6]
that the hypercohomology of $\cI\cS_X$ for conifolds $X$ provides the correct count of massless $3$-branes in type IIB string theory.

\noindent{(iii)} Examples of isolated hypersurface singularities whose monodromy is semi-simple in the eigenvalue $1$ include those for which the monodromy is semi-simple, e.g., weighted homogeneous singularities.

\noindent{(iv)} The splitting $\psi_\pi\bQ_{\wti{X}}[n] \simeq \cI\cS_X \oplus \cC$ for the deformation $\pi:{\wti{X}} \ra S$ of $X$ from  Theorem \ref{thmIS} should be viewed as ``mirroring" the splitting \be Rf_* \bQ_{\wti{X}}[n]=\cI\cC_X\oplus \ \{\rm contributions \ from \ singularities \} \ee for $f:\wti{X}\ra X$ a resolution of singularities, see \cite{BBD,DM, GM3}. (Recall that the perverse sheaf $\cC$ is supported only on the singular point $x$ of $X$.) This analogy is motivated by the fact that  if $\sp:X_s \to X$ denotes the specialization map (e.g., see \cite{BM} for its construction), then by using a resolution of singularities it can be shown that \be\psi_\pi\bQ_{\wti{X}} =R{\sp}_*\bQ_{X_s}.\ee 
Thus, under our assumptions, we get a splitting: \be R{\sp}_*\bQ_{X_s}[n] \simeq \cI\cS_X \oplus \cC.\ee
Note, however, that the specialization map is only continuous, as opposed to a resolution map, which is algebraic.
Therefore, Theorem \ref{thmIS} provides a validation of the idea that the ``mirror" of the intersection cohomology complex $\cI\cC_X$ is the intersection-space complex $\cI\cS_X$, at least in the case of projective hypersurfaces with an isolated singular point. 
\end{rem}

\begin{proof} {\it of Theorem \ref{thmIS}.}

{\it{(a)}} \ First, since $\wti{X}$ is a  
manifold, we have that
$$
\var : \phi_\pi\bQ_{\wti{X}}[n]\lra \psi_\pi\bQ_{\wti{X}}[n]
$$
is an injection in $Perv(X)$, see the proof of \cite{Ma2}[Lem.3.1]. Thus, by the definition of $\iota$, there is a short exact sequence in the abelian category $Perv (X)$
\begin{equation}\label{eqSES}
0\lra \cC\mathop{\lra}^{\iota} \psi_\pi\bQ_{\wti{X}}[n]\lra \cI\cS_X\mathop{\lra} 0 .
\end{equation}
As $Perv(X)$ is the heart of a t-structure on $D^b_c(X)$, there exists,
according for example to \cite[Remark 7.1.13]{BaTISS},
a unique morphism
$\cI\cS_X \to \cC [1]$ such that
\[ \cC \longrightarrow \psi_\pi \bQ_{\wti{X}}[n] \longrightarrow \cI\cS_X \longrightarrow \cC [1] \]
is a distinguished triangle in $D^b_c (X)$.
Consider its long exact sequence of hypercohomology groups,
\[ \cdots \to \bH^{i} (X;\cC) \to \bH^{i} (X;\psi_\pi \bQ_{\wti{X}} [n]) \to
  \bH^{i} (X;\cI\cS_X) \to
  \bH^{i+1} (X;\cC) 
  \to \cdots \]

By construction (but see also the proof of Lemma \ref{calc} below), we have 
\be
\bH^i(X;\cC)=\cH^i({\cC})_x=
\begin{cases}
0 \ \ \ \ \ \ \ \ \ \ \ \ \ \ \ \ \ \  , \ {\rm if} \ i\ne 0, \\ 
\im (T_x-1) \ , \ {\rm if} \ i=0. 
\end{cases}
\ee
Then, by Theorem \ref{thmBM}, the only thing left to prove for part (a) is the injectivity of 
$$
\cH^0({\cC})_x=\im (T_x-1) \lra \bH^0(X;\psi_\pi\bQ_{\wti{X}}[n])=H^n(X_s;\bQ),
$$
that is, that the variation of a vanishing cocycle, modulo monodromy-invariant ones, is nonzero. This follows from Lemma \ref{lemnv}.

{\it{(b)}} \ Since, by definition, $T-1$ is invertible on the non-unipotent vanishing cycles ${^p}{\phi_{\pi,\neq 1}}$, it follows by the considerations of Section \ref{nv} that 
\be\label{aaa}
\cI\cS_X=\coker\left(\im(N) \overset{\Var}{\lra} {^p}{\psi_{\pi,1}}\bQ_{\wti{X}}[n+1] \right),
\ee
where we use the fact discussed in Section \ref{nv} that on ${^p}{\phi_{\pi,1}}$ the morphisms $T-1$ and $N$ differ by an automorphism, and similarly for $\var$ and $\Var$. 
Therefore, as all objects (and arrows) on the right-hand side of (\ref{aaa}) lift to similar objects (and arrows) in the abelian category $\MHM(X)$ of mixed Hodge modules on $X$, it follows that 
$\cI\cS_X$ underlies (under the forgetful functor $\rat$) a mixed Hodge module $\cI\cS^H_X$ on $X$ defined by
$$\cI\cS^H_X:=\coker\left(\im(N) \overset{\Var}{\lra} {\psi^H_{\pi,1}}\bQ_{\wti{X}}[n+1] \right) \in \MHM(X).$$ (Here we use the notations introduced in Section \ref{nv}.)
In particular, the hypercohomology groups of $\cI\cS_X$ carry canonical mixed Hodge structures.

{\it{(c)}} \ Since $\phi_{\pi}(\bQ_{\wti{X}}[n])$ is supported only on the singular point $\{x\}$, we have  (e.g., as in the proof of Lemma \ref{calc} below) that $\phi_{\pi}(\bQ_{\wti{X}}[n])\simeq i_!i^*\phi_{\pi}(\bQ_{\wti{X}}[n])$. Moreover, the (co)support conditions imply that 
$i^*\phi_{\pi}(\bQ_{\wti{X}}[n]) \simeq H^0(i^*\phi_{\pi}(\bQ_{\wti{X}}[n]))$. So, in view of the isomorphism (\ref{456}), it follows that our assumption on the local monodromy $T_x$ implies that $T-1=0$ on $\phi_{\pi,1}(\bQ_{\wti{X}}[n])$. On the other hand, $T-1$ is an isomorphism on $\phi_{\pi,\neq 1}(\bQ_{\wti{X}}[n])$. Therefore, we get the following:
\be \cC:=\im (T-1 : \phi_\pi \bQ_{\wti{X}}[n]\lra \phi_\pi \bQ_{\wti{X}}[n])=\phi_{\pi,\neq 1}(\bQ_{\wti{X}}[n]).
\ee
Moreover, since the variation morphism $\var$ is an isomorphism on $\phi_{\pi,\neq 1}(\bQ_{\wti{X}}[n])$, it follows that the morphism of perverse sheaves $\cC \to \psi_\pi \bQ_{\wti{X}}[n]$ can be identified with the canonical (split) inclusion $\psi_{\pi,\neq 1} \bQ_{\wti{X}}[n] \hookrightarrow \psi_\pi \bQ_{\wti{X}}[n]$.
Altogether, we obtain that 
\be \cI\cS_X \simeq \psi_{\pi, 1} \bQ_{\wti{X}}[n] , \ee
and the claimed splitting is just the canonical one from (\ref{678}).

{\it{(d)}} \ As it follows from part ({c}), our assumption on the local monodromy operator $T_x$ yields that 
$\cI\cS_X \simeq \psi_{\pi, 1} \bQ_{\wti{X}}[n] $. Since the nearby complex $\psi_{\pi}(\bQ_{\wti{X}}[n])$ is self-dual (e.g., see \cite{Ma, Sa, Sa1}) and the Verdier duality functor respects the splitting (\ref{678}), it follows that $\cI\cS_X$ is also self-dual. 
From the self-duality of $\cI\cS_X$,  one obtains readily  a non-degenerate paring 
$$
\bH^{-i}(X;\cI\cS_X) \times \bH^i_c(X;\cI\cS_X)\ra \bQ .
$$
But since $X$ is compact, $\bH^i_c(X;\cI\cS_X)=\bH^i(X;\cI\cS_X)$. 

\end{proof}

\begin{rem} 
\noindent{(i)} The proof of part (a) shows in fact that for any $i \notin \{n, 2n\}$, we have isomorphisms
\[ \bH^i(X; \cI\cS_X[-n]) \simeq H^i(X_s;\bQ). \]
Then by Theorem \ref{thmBM}, for any $i \notin \{n, 2n\}$ we have (abstract) isomorphisms:
\[ H^i(X_s;\bQ) \simeq H^i(IX;\bQ). \]

\noindent{(ii)} A topological interpretation of the splitting and self-duality of parts ({c}) and (d) of the above theorem, will be given in Sections \ref{split} and \ref{sd}, respectively, by using the theory of zig-zags (as recalled in Section \ref{zz}).
\end{rem}

\br {\it Purity and Hard-Lefschetz}\newline
It follows from the proof of part $(b)$ of Theorem \ref{thmIS} that the intersection-space complex $\cI\cS_X$ is a perverse sheaf underlying the mixed Hodge module $\cI\cS^H_X$. While this mixed Hodge module is not in general pure, it is however possible for its hypercohomology $\bH^*(X;\cI\cS_X)$ to be pure and, moreover, to satisfy the Hard Lefschetz theorem. It is known that such statements are always true for the ``mirror theory", that is, intersection cohomology (e.g., see \cite{DM} and the references therein).
We give here a brief justification of these claims, by making use of results from the recent preprint \cite{DMSS}. Recall that, under the assumption of semisimplicity in the eigenvalue $1$ for the local monodromy operator $T_x$, we show in part $({c})$ of Theorem \ref{thmIS} that $ \cI\cS_X \simeq {^p}\psi_{\pi, 1} \bQ_{\wti{X}}[n+1]$. Next, following \cite{DMSS}[Thm.1.1], we note that the weight filtration of the mixed Hodge structure on $\bH^i({^p}\psi_{\pi, 1}):=\bH^i(X; {^p}\psi_{\pi, 1} \bQ_{\wti{X}}[n+1])$ is (up to a shift) the monodromy filtration of the nilpotent endomorphism $N$ acting on $\bH^i({^p}\psi_{\pi, 1})$. So $\bH^i({^p}\psi_{\pi, 1})$ or, under the local monodromy assumption, $\bH^i(X;\cI\cS_X)$, is a pure Hodge structure if and only if $N=0$ or, equivalently, the monodromy $T$ acting on $\bH^i(^p\psi_{\pi, 1})$ is semisimple. 
Note that by our calculations in Lemma \ref{lemnv}, this semisimplicity is equivalent to the condition that the monodromy $T$ acting on $^p\psi_{s,1}({^p\cH}^j(R\pi_*\bQ_{\wti{X}}[n+1]))$ is semisimple for all perverse sheaves ${^p\cH}^j(R\pi_*\bQ_{\wti{X}}[n+1])$ on the disc  $S$. Moreover, if this is the case, then one can show as in \cite{DMSS}[Thm.1.4] that the Hard Lefschetz theorem also holds for the hypercohomology groups $\bH^i(X;\cI\cS_X)$. To sum up, the above discussion implies that, if besides the ``local" semisimplicity assumption for the eigenvalue $1$ of $T_x$, we also require a ``global" semisimplicity property for the eigenvalue $1$ (i.e., the monodromy $T$ acting on $\bH^i(X;{^p}\psi_{\pi, 1} \bQ_{\wti{X}}[n+1])$ is semisimple), then the hypercohomology groups  $\bH^i(X;\cI\cS_X)$ carry pure Hodge structures satisfying the hard Lefschetz theorem.   
Finally, we should also mention that  the above technical assumption of ``global semisimplicity for the eigenvalue $1$" is satisfied if the following geometric condition holds: 
$\wti{X}$ carries a $\bC^*$-action so that $\pi:\wti{X} \to S$ is equivariant with respect to the weight $d$ action of $\bC^*$ on $S$ ($d>0$), see \cite{DMSS}[Sect.2] for details.
\er


\subsection{On the splitting of nearby cycles}\label{split}
In this section, we use the theory of ziz-zags to provide a topological interpretation of the splitting statement of part $({c})$ of our main Theorem \ref{thmIS}. More precisely, we will consider the zig-zags associated to the defining sequence for $\cI\cS_X$, 
\be\label{eight} 0\lra \cC \overset{\iota}{\lra} \psi_\pi\bQ_{\wti{X}}[n]\lra \cI\cS_X\mathop{\lra} 0, \ee
in order to show that this short exact sequence splits. 

Let us denote as before by $i:\{x\} \hookrightarrow X$ and $j:X^{\circ}:=X\setminus \{x\} \hookrightarrow X$ the closed and resp. open embeddings. 
Recall from Example \ref{zzn} that the zig-zag associated to the perverse sheaf $\psi_{\pi}(\bQ_{\wti{X}}[n])$ consists of the triple $(\bQ_{X^{\circ}}[n],H^n(F,L;\bQ),H^n(F;\bQ))$, together with the exact sequence
\be\label{zzne2}
H^{n-1}(L;\bQ) \lra H^n(F,L;\bQ) \lra H^n(F;\bQ) \lra H^n(L;\bQ).
\ee
Similarly, we have the following:
\begin{lem}\label{calc}
The zig-zag for $\cC$ consists of the triple $(0,A,B)$ together with the exact sequence
\be\label{zzne2}
0 \lra A \lra B \lra 0,
\ee
where 
\be\label{six} A=H^0(i^!\cC) \simeq \im\Big(T_x-1:H^n(F,L;\bQ) \to H^n(F,L;\bQ)\Big),\ee and 
\be\label{seven} B=H^0(i^*\cC) \simeq \im\Big(T_x-1:H^n(F;\bQ) \to H^n(F;\bQ)\Big).\ee
\end{lem}
\begin{proof}
Since $j^* \cC=0,$ we have
\[ H^{-1} (i^* Rj_* j^* \cC)=0,~ H^{0} (i^* Rj_* j^* \cC)=0, \]
that is, the two outermost terms of the zig-zag exact sequence vanish.
Let us determine the inner terms $A$ and $B$.
As $\cC$ is supported on the singular point $x$, we first get that $j^*\cC \simeq j^!\cC \simeq 0$. By using the attaching triangles, we get that  \be \cC \simeq i_*i^*\cC \simeq  i_*i^!\cC.\ee So by applying $i^*$ and using $i^*i_* \simeq id$, we get that $i^*\cC \simeq i^!\cC$. Moreover, the (co)support conditions 
for $\cC$ imply that $i^*\cC \in Perv(\{x\})$, i.e., $H^i(i^*\cC)=0$ for all $i \neq 0$. Hence, there is a quasi-isomorphism $i^*\cC \simeq H^0(i^*\cC)$. Similar considerations apply to the perverse sheaf $\phi_{\pi}(\bQ_{\wti{X}}[n])$, which is also supported only on the singular point $x$. Recall now that $\cC$ is defined by the sequence of perverse sheaves supported on $\{x\}$:
\be\label{five}
\phi_{\pi}(\bQ_{\wti{X}}[n]) \overset{T-1}{\twoheadrightarrow} \cC \hookrightarrow \phi_{\pi}(\bQ_{\wti{X}}[n]).
\ee
As $i^*$ is always right $t$-exact and $i^!$ is left t-exact, it follows that 
on objects supported over the point $x$, $i^*$ is t-exact and thus
we get the following sequence of rational vector spaces
\be
H^0(i^*\phi_{\pi}(\bQ_{\wti{X}}[n])) \overset{T-1}{\twoheadrightarrow} H^0(i^*\cC) \hookrightarrow H^0(i^*\phi_{\pi}(\bQ_{\wti{X}}[n])).
\ee
In view of the identification
$H^0 (i^* \phi_\pi (\bQ_{\wti{X}}[n])) = H^n (F;\bQ)$, the above sequence  proves (\ref{seven}). 
Similarly, by applying $i^!$ to (\ref{five}), we obtain (\ref{six})
in view of the identification
$H^0 (i^! \phi_\pi (\bQ_{\wti{X}}[n])) = H^n (F,L;\bQ)$.

\end{proof}

Therefore, the map $\iota$ of (\ref{eight}) corresponds under the zig-zag functor to the map of zig-zags $\iota: \cZ({\cC}) \to \cZ(\psi_{\pi}(\bQ_{\wti{X}}[n]))$ given as:
\be\label{555}
\begin{CD}
0 @>>> A @>\beta'>> B @>>> 0 \\
@VVV @V{\iota_a} VV @V{\iota_b}VV @VVV \\
H^{n-1}(L;\bQ) @>\alpha>> H^n(F,L;\bQ) @>\beta>> H^n(F;\bQ) @>\gamma>> H^n(L;\bQ) ,
\end{CD}
\ee
with $A$ and $B$ defined as in (\ref{six}) and (\ref{seven}).

Assume from now on that the following condition is satisfied: 
\begin{ass}\label{ass}
The local monodromy $T_x$ at the singular point is semi-simple in the eigenvalue $1$.
\end{ass}

The following result will be needed in Proposition \ref{200} below.
\bl
Under the assumption \ref{ass}, the vertical homomorphisms $\iota_a$ and $\iota_b$ of the above diagram (\ref{555}) are injective.
\el

\begin{proof}
By the commutativity of the middle square of (\ref{555}), the injectivity of
$\iota_a$ will follow from the injectivity of $\iota_b$, since $\beta'$ is an
isomorphism. Nevertheless, we shall also establish the injectivity of $\iota_a$
directly, and then go on to prove that $\iota_b$ is injective.
From the computations above, it follows that  $\iota_a$ is the homomorphism 
\be \iota_a: H^0(i^!\cC) \to H^0(i^!\psi_{\pi}(\bQ_{\wti{X}}[n]))\ee
obtained by restricting the variation morphism 
\be\label{666} \var: H^0(i^!\phi_{\pi}(\bQ_{\wti{X}}[n])) \lra H^0(i^!\psi_{\pi}(\bQ_{\wti{X}}[n])) \ee
to the subspace $$H^0(i^!\cC)=\im\Big(T-1: H^0(i^!\phi_{\pi}(\bQ_{\wti{X}}[n])) \to H^0(i^!\phi_{\pi}(\bQ_{\wti{X}}[n])) \Big) \subset H^0(i^!\phi_{\pi}(\bQ_{\wti{X}}[n])).$$
We claim that (\ref{666}) is a vector space isomorphism, which in turn yields that $\iota_a$ is injective. 
To prove the claim, make $K=\bQ_{\wti{X}}[n]$ in  (\ref{varnew}). We get a distinguished triangle
\be\label{999} \phi_\pi \bQ_{\wti{X}}[n] \mathop{\lra}^{\var} \psi_\pi \bQ_{\wti{X}}[n] \lra t^! \bQ_{\wti{X}}[n+2] \mathop{\lra}^{[+1]} \ee
with  $t:X\hookrightarrow \wti{X}$ denoting the inclusion.  By applying the functor $i^!$ and taking the cohomology of the resulting triangle, we see that the variation morphism (\ref{666}) fits into an exact sequence:
\be\label{777} \cdots \lra H^{n+1}(i^! t^! \bQ_{\wti{X}}) \lra H^0(i^!\phi_{\pi}(\bQ_{\wti{X}}[n])) \mathop{\lra}^{\var} H^0(i^!\psi_{\pi}(\bQ_{\wti{X}}[n])) \lra H^{n+2}(i^! t^! \bQ_{\wti{X}}) \lra \cdots \ee
If we now let $e:=t \circ i :\{x\} \hookrightarrow \wti{X}$ be the inclusion of the point $x$ in the ambient space $\wti{X}$, then for any $k \in \bZ$ we have that
\be H^{k}(i^! t^! \bQ_{\wti{X}})=H^{k}(e^! \bQ_{\wti{X}}) \simeq H^k_c(B^{2n+2}_x;\bQ), \ee
for $B^{2n+2}_x$ a small (euclidian) ball around $x$ in $\wti{X}$. Finally, $H^k_c(B^{2n+2}_x;\bQ) = 0$ for all $k \neq 2n+2$, which by (\ref{777}) proves our claim.

Similarly,  $\iota_b$ is the homomorphism 
\be \iota_b: H^0(i^*\cC) \to H^0(i^*\psi_{\pi}(\bQ_{\wti{X}}[n]))\ee
obtained by restricting the variation morphism 
\be\label{888} \var: H^0(i^*\phi_{\pi}(\bQ_{\wti{X}}[n])) \lra H^0(i^*\psi_{\pi}(\bQ_{\wti{X}}[n])) \ee
to the subspace $$H^0(i^*\cC)=\im\Big(T-1: H^0(i^*\phi_{\pi}(\bQ_{\wti{X}}[n])) \to H^0(i^*\phi_{\pi}(\bQ_{\wti{X}}[n])) \Big) \subset H^0(i^*\phi_{\pi}(\bQ_{\wti{X}}[n])).$$ 
So in order to show that $\iota_b$ is injective, we need to prove that 
\be\label{ccc} \Ker(\var) \cap \im (T-1)=\{0\}.
\ee
We shall first show that 
\be\label{bbb}
\Ker(T-1)=\Ker(\var).
\ee
As $T-1=\can \circ \var$, we have that $\Ker(\var) \subset \Ker(T-1)$. So to establish the equality (\ref{bbb}), it suffices to show that both kernels have the same dimension. 

By applying the functor $i^*$ to the distinguished triangle (\ref{999}) and taking the cohomology of the resulting triangle, we see that the variation morphism (\ref{888}) fits into an exact sequence:
\be\label{000} 
 0=H^{-1}(i^*\psi_{\pi}(\bQ_{\wti{X}}[n])) \to
 H^{n+1}(i^*t^! \bQ_{\wti{X}}) \to H^0(i^*\phi_{\pi}(\bQ_{\wti{X}}[n])) \mathop{\lra}^{\var} H^0(i^*\psi_{\pi}(\bQ_{\wti{X}}[n]))
 \to \cdots \ee
Moreover, using the isomorphisms  $H^k(i^*\phi_{\pi}(\bQ_{\wti{X}}[n])) \simeq H^{k+n}(F;\bQ) \simeq H^k(i^*\psi_{\pi}(\bQ_{\wti{X}}[n]))$ for $k+n>0$, 
together with the fact that $F$ is $(n-1)$-connected and $n>2$, we get from (\ref{000}) that 
\be \Ker\Big(  H^0(i^*\phi_{\pi}(\bQ_{\wti{X}}[n])) \mathop{\lra}^{\var} H^0(i^*\psi_{\pi}(\bQ_{\wti{X}}[n])) \Big) \simeq H^{n+1}(i^*t^! \bQ_{\wti{X}}). \ee
We claim that 
\be\label{0000} H^{n+1}(i^*t^! \bQ_{\wti{X}}) \simeq \Ker\Big( T-1: H^0(i^*\phi_{\pi}(\bQ_{\wti{X}}[n])) \to H^0(i^*\phi_{\pi}(\bQ_{\wti{X}}[n])) \Big) \ee
which, in view of the above discussion, yields the equality (\ref{bbb}). In order to prove (\ref{0000}), let us first denote by $s:\wti{X} \setminus X \hookrightarrow \wti{X}$ the open embedding complementary to $t:X \hookrightarrow \wti{X}$. Note that since $t$ is closed, we have that $t^*t_* \simeq id$. Therefore, by using the inclusion $e:=t \circ i:\{x\} \hookrightarrow {\wti{X}}$ as above, we obtain:
\be 
\begin{split}
H^{n+1}(i^*t^! \bQ_{\wti{X}})  \simeq H^{n+1}(i^*t^*t_*t^! \bQ_{\wti{X}}) \simeq H^{n+1}(e^*t_*t^! \bQ_{\wti{X}}) \simeq \cH^{n+1}(t_*t^! \bQ_{\wti{X}})_x.
\end{split}
\ee
Moreover, by using the attaching triangle \be t_*t^!  \lra id \lra s_*s^* \overset{[1]}{\lra},\ee
as $\cH^n(\bQ_{\wti{X}})_x=0=\cH^{n+1}(\bQ_{\wti{X}})_x$, we get that:
\be 
\begin{split}
\cH^{n+1}(t_*t^! \bQ_{\wti{X}})_x & \simeq \cH^n(s_*s^*\bQ_{\wti{X}})_x \simeq \bH^n(B^{2n+2}_x;s_*\bQ_{\wti{X} \setminus X}) \\ & \simeq H^n(B^{2n+2}_x \setminus X;\bQ) \simeq H^n(S^{2n+1}_x \setminus L;\bQ),
\end{split}
\ee
where $B^{2n+2}_x$ denotes as before a small enough euclidian ball around $x$ in $\wti{X}$, having as boundary the sphere $\partial B^{2n+2}_x=S^{2n+1}_x$. Finally, by the cohomology Wang sequence dual to (\ref{wang}), we have that
\be H^n(S^{2n+1}_x \setminus L;\bQ) \simeq \Ker\Big( T_x-1:H^n(F;\bQ) \lra H^n(F;\bQ) \Big), \ee
which finishes the proof of (\ref{0000}).

Therefore, by (\ref{ccc}) and (\ref{bbb}), the claim on the injectivity of $\iota_b$ is equivalent with
\be
\Ker(T-1) \cap \im(T-1)=\{0\},
\ee
which follows from our assumption \ref{ass} on $T_x$.

\end{proof}

We can now prove the following
\bp\label{200}
Under the assumption \ref{ass}, the zig-zag sequence corresponding to (\ref{eight}) splits,  i.e., there exists a zig-zag morphism   $$\sigma :\cZ(\psi_{\pi}(\bQ_{\wti{X}}[n])) \to \cZ({\cC})$$ so that $\sigma \circ \iota = id_{\cZ({\cC})}$.
\ep

\begin{proof} This amounts to defining vector space homomorphisms $\sigma_a:H^n(F,L;\bQ) \to A$ and $\sigma_b:H^n(F;\bQ) \to B$ so that $\sigma_a \circ \iota_a=id_A$, $\sigma_b \circ \iota_b=id_B$ and $\sigma_b \circ \beta=\beta' \circ \sigma_a$. Note that then automatically
$\sigma_a \circ \alpha =0$, as
\[ \sigma_a \alpha = (\beta')^{-1} \beta' \sigma_a \alpha =
  (\beta')^{-1} \sigma_b \beta \alpha =0 \]
by exactness.

In order to define $\sigma_a$, let us first describe a basis for the vector space $H^n(F,L;\bQ)$.
Let $\{v_1, \cdots, v_m\}$ be a basis for $\Ker(\beta)$ and $\{a_1, \cdots, a_p\}$ a basis for $A$. Let $A':=\im(\iota_a) \subset H^n(F,L;\bQ)$. As $\beta|_{A'}$ is injective, we have that $\Ker(\beta) \cap A'=0$. Therefore, the set of vectors $\{v_1, \cdots, v_m, \iota_a(a_1), \cdots, \iota_a(a_p\}$ is  linearly independent  in $H^n(F,L;\bQ)$. Next, consider the short exact sequence
$$
0 \lra \Ker(\beta) \lra H^n(F,L;\bQ) \overset{q}{\lra} Q:=H^n(F,L;\bQ)/{\Ker(\beta)} \lra 0
$$
and note that $\Ker(\beta)=\Ker(q)$. Moreover, since $q|_{A'}$ is a monomorphism, it follows that $\{ q\iota_a(a_1), \cdots q\iota_a(a_p) \}$ are linearly independent set of vectors in $Q$.  Extend this collection of vectors to a basis
$\{ q\iota_a(a_1), \cdots q\iota_a(a_p), q_1, \cdots , q_k \}$  of $Q$, and define a homomorphism $s:Q \to 
H^n(F,L;\bQ)$ by the rule: $s(q\iota_a(a_j))=\iota_a(a_j)$, for $j=1,\cdots,p$, and by choosing arbitrary values in $q^{-1}(q_j)$ for $s(q_j)$, $j=1,\cdots,k$. Then it is easy to see that $q \circ s=id_Q$, hence
$$H^n(F,L;\bQ)=\Ker(\beta) \oplus \im(s).$$  
So a basis for $H^n(F,L;\bQ)$ can be chosen as 
$$\{v_1, \cdots, v_m,\iota_a(a_1), \cdots, \iota_a(a_p), s(q_1), \cdots, s(q_k) \}.$$
We can now define the map $\sigma_a:H^n(F,L;\bQ) \to A$ as follows: $\sigma_a(v_j)=0$ for all $j=1,\cdots,m$, $\sigma_a(\iota_a(a_j))=a_j$ for all $j=1,\cdots,p$, and $\sigma_a(s(q_j))=0$ for all $j=1,\cdots,k$. Clearly, we also have $\sigma_a \circ \iota_a=id_A$.

In order to define $\sigma_b$, we first describe a basis for $H^n(F;\bQ)$ as follows. With the above definition of $Q$, let $\bar{\beta}:Q \hookrightarrow H^n(F;\bQ)$ be the canonical inclusion so that $\beta=\bar{\beta} \circ q$. Then 
$$\{ \bar{\beta}q\iota_a(a_1), \cdots \bar{\beta}q\iota_a(a_p), \bar{\beta}(q_1), \cdots , \bar{\beta}(q_k) \}$$ is a set of linearly independent vectors in $H^n(F;\bQ)$. Note that $\bar{\beta}q\iota_a(a_j)=\beta\iota_a(a_j)=\iota_b \beta'(a_j)$. Also, since $\beta'$ is an isomorphism, $\{\beta'(a_1), \cdots, \beta'(a_p)\}$ is a basis for $B$. Extend the above collection of vectors to a basis $$\{\iota_b \beta'(a_1),\cdots,\iota_b \beta'(a_p), \bar{\beta}(q_1), \cdots , \bar{\beta}(q_k), e_1,\cdots,e_l \}$$ for $H^n(F;\bQ)$. We can now define $\sigma_b:H^n(F;\bQ) \to B$ as follows: $\sigma_b(\iota_b \beta'(a_j))=\beta'(a_j)$ for all $j=1,\cdots,p$, $\sigma_b(\bar{\beta}(q_j)=0$ for all $j=1,\cdots,k$, and $\sigma_b(e_j)=0$ for all $j=1,\cdots,l$. Then we have by definition that $\sigma_b \circ \iota_b=id_B$. 

Finally, we check the commutativity relation $\sigma_b \circ \beta=\beta' \circ \sigma_a$ on basis elements of $H^n(F,L;\bQ)$. For any index $j$ in the relevant range, we have: 
\bn
\item[(i)] $\sigma_b \beta(v_j)=0=\beta' \sigma_a (v_j)$ by the choice of $v_j \in \Ker(\beta)$ and by the definition of $\sigma_a$.
\item[(ii)] $\sigma_b \beta(\iota_a(a_j))=\sigma_b \iota_b \beta'(a_j)=\beta'(a_j)=\beta' \sigma_a (\iota_a(a_j))$, where the second equality follows from $\sigma_b \circ \iota_b=id_B$, and the 
 third equality is a consequence of $\sigma_a \iota_a = id_A$.
\item[(iii)] $\sigma_b \beta(s(q_j))=\sigma_b\bar{\beta} q(s(q_j))=\sigma_b(\bar{\beta}(q_j)=0$, and $\beta' \sigma_a(s(q_j))=\beta'(0)=0.$
\en

\end{proof}

As a consequence, we obtain the following
\bc \label{corsplitting}
Under the assumption \ref{ass}, there is a 
splitting of the short exact sequence (\ref{eight}), hence:
$$\psi_\pi\bQ_{\wti{X}}[n]\simeq \cI\cS_X \oplus \cC.$$
\ec

\begin{proof}
Since $\beta'$ is an isomorphism, by the second part of Theorem \ref{MV} the existence of such a splitting for (\ref{eight}) is equivalent to a splitting at the zig-zag level. So the result follows from Proposition \ref{200}.

\end{proof}

\subsection{On self-duality}\label{sd}
In this section, we use the theory of ziz-zags to provide a topological interpretation of the self-duality of the intersection space complex $\cI\cS_X$ shown in part $({d})$ of our main Theorem \ref{thmIS}.

We will still be working under the assumption \ref{ass}. Let $\cD$ be the Verdier dualizing functor in $D^b_c(X)$. Since $X$ is compact, for any  $K\in D^b_c(X)$  there is a non-degenerate pairing
\be\label{ten}
\bH^{-i}(X;\cD K)\otimes \bH^i(X;K)\lra \bQ.
\ee
Recall that the Verdier duality functor $\cD$ fixes the category $Perv(X)$, i.e., it sends perverse sheaves to perverse sheaves.

We will use part $(b)$ of Theorem \ref{MV}  to show that the intersection-space complex $\cI\cS_X$ is self-dual in $Perv(X)$, i.e., there is an isomorphism $$\cI\cS_X \simeq \cD(\cI\cS_X).$$ More precisely, we will use a  duality functor $\cD_{\cZ}$ in the category $Z(X,x)$ (see \cite{R}[Sect.4]) to show that the zig-zag associated to $\cI\cS_X$ is self-dual. Then  Theorem \ref{MV}  will generate a corresponding 
self-duality isomorphism for $\cI\cS_X$.

\bigskip

Let us first describe the duality functor $\cD_{\cZ}$ on the zig-zag category $Z(X,x)$. Recall that $X$ is a $n$-dimensional  projective hypersurface with only one isolated singularity $x$, and $i:\{x\} \hookrightarrow X$ and $j:X^{\circ}=X\setminus \{x\} \hookrightarrow X$ denote the respective inclusion maps. The following identities are well-known: 
\be\label{11} i^*Rj_* \simeq i^!Rj_! [1] \ , \ \cD i^*\simeq i^!\cD \ , \ \cD j^*\simeq j^*\cD \ , \ \cD i_*\simeq i_!\cD\simeq i_*\cD \ , \ \cD Rj_*\simeq Rj_!\cD.
\ee

Let us now fix an object  $K \in Perv(X)$ with associated zig-zag:
\be\label{300}
\begin{CD}
H^{-1}(i^*Rj_*j^* K) @>\alpha>> H^0(i^! K) @>\beta>> H^0(i^* K) @>\gamma>> H^{0}(i^*Rj_*j^* K), 
\end{CD}
\ee
and recall that $j^* K \simeq \cL[n]$, for some local system $\cL$ with finite dimensional stalks on $X^{\circ}$.
The zig-zag $\cZ(\cD K)$  associated to the Verdier dual $\cD K$ is then defined  by the triple
$$(j^*\cD K, H^0(i^!\cD K), H^0(i^*\cD K)),$$ together with the exact sequence:
\be
\begin{CD}
H^{-1}(i^*Rj_*j^*\cD K) @>u>> H^0(i^!\cD K) @>v>> H^0(i^*\cD K) @>w>> H^{0}(i^*Rj_*j^*\cD K).
\end{CD}
\ee
By using (\ref{ten}), (\ref{11}) and Remark \ref{rem1}, $\cZ(\cD K)$ is then isomorphic to the triple 
$$(\cD j^*K, H^0(i^* K)^\vee, H^0(i^! K)^\vee)=(\cL^\vee[n], H^0(i^* K)^\vee, H^0(i^! K)^\vee),$$ together with the exact sequence:
\be \label{seq.dualzigzag}
\begin{CD}
H^{0}(i^*Rj_*j^* K)^\vee @>\gamma^\vee>> H^0(i^* K)^\vee @>\beta^\vee>> H^0(i^! K)^\vee @>\alpha^\vee>> H^{-1}(i^*Rj_*j^* K)^\vee,
\end{CD}
\ee
where for a vector space $V$, local system $\cL$ and  homomorphism $f$, we denote by $V^\vee$, $\cL^\vee$, $f^\vee$ their respective duals. In other words, the zig-zag of the dual complex $\cD K$ is obtained by ``dualizing" the zig-zag of $K$, i.e., by considering the corresponding dual vectors spaces and dual maps in (\ref{300}). 

\bd \label{def.dualofzigzag}
The above operation of dualizing a zig-zag defines a duality functor on $Z(X,x)$, denoted by $\cD_{\cZ}$, which is compatible with the Verdier dual on $Perv(X)$, i.e., we define
\be\label{12}
\cD_{\cZ}(\cZ (K))=(\cL^\vee[n], H^0(i^* K)^\vee, H^0(i^! K)^\vee)
\ee
together with the sequence (\ref{seq.dualzigzag}).
\ed

\br In terms of the local system $\cL:=j^*K[-n]$ associated to $K \in Perv(X)$, the zig-zags for $K$ and $\cD K$ correspond respectively to the long exact sequences:
\be
\begin{CD}
\cZ(K): \ \ \ \ H^{n-1}(i^*Rj_*\cL) @>\alpha>> H^0(i^! K) @>\beta>> H^0(i^* K) @>\gamma>> H^{n}(i^*Rj_*\cL), 
\end{CD}
\ee
and 
\be
\begin{CD}
\cZ(\cD K): \ \ H^{n-1}(i^*Rj_*\cL^\vee) @>\gamma^\vee>> H^0(i^* K)^\vee @>\beta^\vee>> H^0(i^! K)^\vee @>\alpha^\vee>> H^{n}(i^*Rj_*\cL^\vee).
\end{CD}
\ee
Indeed, for any $r \in \bZ$, we have
\be
\begin{split}
H^{r}(i^*Rj_*j^* K)^\vee & \simeq H^{n+r}(i^*Rj_*\cL)^\vee \overset{(\ref{ten})}{\simeq} H^{-n-r}(\cD (i^*Rj_*\cL)) \simeq H^{-n-r}(i^!Rj_!\cD(\cL))) \\ &\simeq H^{-n-r}(i^!Rj_!\cL^\vee[2n]) \simeq H^{n-r}(i^!Rj_!\cL^\vee) \overset{(\ref{11})}{\simeq} H^{n-r-1}(i^*Rj_*\cL^\vee).
\end{split}
\ee
\er

We can now prove the following
\bp\label{13} Under the assumption \ref{ass}, the zig-zag associated to the intersection-space complex $\cI\cS_X$ is self-dual in the  category $Z(X,x)$.
\ep

\begin{proof}
Applying $\cZ$ to the short exact sequence
\[ 0\longrightarrow \cC \longrightarrow \psi_\pi \bQ_{\wti{X}}[n]
  \longrightarrow \cI\cS_X \longrightarrow 0, \]
we obtain the commutative diagram
\begin{equation} \label{dia3by3}
\xymatrix{
0 \ar[d] & 0 \ar[d] & 0 \ar[d] & 0 \ar[d] \\
0 \ar[d] \ar[r] & A \ar[d] \ar[r]^{\simeq}_{\beta'} & B \ar[d]_{\iota_b} \ar[r] & 0 \ar[d] \\
H^{n-1}(L;\bQ) \ar[d] \ar[r]_{\alpha} & H^n (F,L;\bQ) \ar[d] \ar[r]_{\beta} & H^n (F;\bQ) 
    \ar[d]_{p_b} \ar@{->>}[r]_{\gamma} & 
     H^n (L;\bQ) \ar[d]_{p_w} \\
V \ar[d] \ar[r]_{\wti{\alpha}} & A' \ar[d] \ar[r]_{\wti{\beta}} & B' \ar[d] \ar[r]_{\wti{\gamma}} & 
     W \ar[d] \\
0 & 0  & 0 & 0, \\
} 
\end{equation}
where
\[ A' = H^0 (i^! \cI\cS_X),~ B' = H^0 (i^* \cI\cS_X) \]
and
\[ V = H^{-1}(i^* Rj_* j^* \cI\cS_X) \simeq H^{n-1}(L;\bQ),~
   V = H^{0}(i^* Rj_* j^* \cI\cS_X) \simeq H^{n}(L;\bQ), \]
using that $j^* \cI\cS_X = \bQ_{X^{\circ}}[n]$.
The rows of this diagram are known to be exact, but we do \emph{not} know a priori that the
columns are exact, since $\cZ$ is generally not an exact functor.
However, using the splitting obtained in Corollary \ref{corsplitting}, there is a commutative diagram
\[ \xymatrix{
0 \ar[r] & \cC \ar@{=}[d] \ar[r] & \psi_\pi \bQ_{\wti{X}}[n] \ar[d]_{\simeq} \ar[r]
  & \cI\cS_X \ar@{=}[d] \ar[r] & 0\\
0 \ar[r] & \cC \ar[r]^>>>>{incl} & \cC \oplus \cI\cS_X \ar[r]^{proj}
  & \cI\cS_X \ar[r] & 0,
} \]
where the bottom row is the standard short exact sequence associated to the direct sum, i.e.
the map labelled $incl$ is the standard inclusion and the map labelled $proj$ is the standard
projection. Since the zig-zag category is additive and $\cZ$ an additive functor, we obtain
commutative diagrams such as e.g.
\[ \xymatrix{
0 \ar[r] & A \ar@{=}[d] \ar[r] & H^n (F,L;\bQ) \ar[d]_{\simeq} \ar[r]
  & A' \ar@{=}[d] \ar[r] & 0\\
0 \ar[r] & A \ar[r]^>>>>>>{incl} & A \oplus A' \ar[r]^{proj}
  & A' \ar[r] & 0,
} \]
This proves that the second column of diagram (\ref{dia3by3}) is exact.
Similarly, all other columns of (\ref{dia3by3}) are exact.
We claim that $\im \beta = \im \iota_b$ in $H^n (F;\bQ)$. To establish
this, we observe first that 
\[ \im \iota_b = \im (\iota_b \beta') = \im (\beta \iota_a )\subset \im \beta. \]
Thus it remains to show that $\beta$ and $\iota_b$ have the same rank.
By the Wang sequence, we have the following cohomological version of \cite{S}[Prop.2.2]: $$\im(\beta) \simeq \im(T_x-1) \simeq B.$$
This proves the claim.
Diagram (\ref{dia3by3}) shows that $\wti{\gamma}$ is surjective.
Suppose that $\wti{\gamma}(b')=0$. Then there is a class $x\in H^n (F;\bQ)$
with $p_b (x)=b'$ and $p_w (\gamma (x))=\wti{\gamma} (p_b (x))=\wti{\gamma} (b')=0$.
As $p_w$ is an isomorphism, $\gamma (x)=0$. So $x$ is in $\im \beta = \im \iota_b$,
which implies that $b'=p_b (x)=0.$ We have shown that $\wti{\gamma}$ is injective,
hence an isomorphism. This implies that $\wti{\beta}$ is the zero map
and $\wti{\alpha}$ is an isomorphism. So the zig-zag $\cZ(\cI\cS_X)$ for $\cI\cS_X$ looks like:
\be\label{zzis}
\begin{CD}
 H^{n-1}(L;\bQ) @>\simeq>> H^n(F,L;\bQ)/A @>0>> H^n(F;\bQ)/B @>\simeq>> H^n(L;\bQ).
\end{CD}
\ee
By Definition \ref{def.dualofzigzag}, the zig-zag $\cZ(\cD(\cI\cS_X))$ for the Verdier dual of $\cI\cS_X$ is given by $\bQ_{X^{\circ}}[n]$ on $X^{\circ}$, together with the exact sequence (dualizing the corresponding one for $\cZ(\cI\cS_X))$:
\be
\begin{CD}
 H^{n}(L;\bQ)^\vee @>\simeq>>  (H^n(F;\bQ)/B)^\vee @>0>>  (H^n(F,L;\bQ)/A)^\vee @>\simeq>> H^{n-1}(L;\bQ)^\vee.
\end{CD}
\ee
An isomorphism between the zig-zags $\cZ(\cI\cS_X)$ and $\cZ(\cD(\cI\cS_X))$ can now be defined 
as follows: The orientation of $X^\circ$ specifies a self-duality isomorphism
\[ d: j^* \cI\cS_X = \bQ_{X^\circ}[n] \cong \cD \bQ_{X^\circ}[n] =
   \cD j^* \cI\cS_X = j^* \cD \cI\cS_X. \]
The induced isomorphism
\[ d_*: H^{-1} (i^* Rj_* j^* \cI\cS_X) \cong H^{-1} (i^* Rj_* j^* \cD \cI\cS_X) \] 
(similarly on $H^0$) corresponds to
the non-degenerate Poincar\'e duality pairing $H^{n-1}(L;\bQ) \otimes H^{n}(L;\bQ) \to \bQ$ for  the link $L$.
It determines uniquely isomorphisms
$H^n (F,L;\bQ)/A \cong (H^n (F;\bQ)/B)^\vee$ and
$H^n (F;\bQ)/B \cong (H^n (F,L;\bQ)/A)^\vee$ such that
{\small
\[ \begin{CD}
H^{-1}(i^*Rj_* j^* \cI\cS_X) @>\simeq>> H^n (F,L;\bQ)/A @>0>> H^n (F;\bQ)/B @>\simeq>> 
    H^{0}(i^*Rj_* j^* \cI\cS_X)\\ 
@VV{d_*}V @VVV @VVV @VV{d_*}V \\
H^{-1}(i^*Rj_* j^* \cD \cI\cS_X) @>\simeq>> (H^n (F;\bQ)/B)^\vee @>0>> 
  (H^n (F,L;\bQ)/A)^\vee @>\simeq>> 
  H^{0}(i^*Rj_* j^* \cI\cS_X)
\end{CD} \]
}
commutes.

\end{proof}

\bc \label{corselfdual}
Under the assumption \ref{ass}, the perverse sheaf $\cI\cS_X$ is self-dual in $Perv(X)$.
\ec

\begin{proof} In Prop.\ref{13}, we have constructed an isomorphism in 
$\Hom(\cZ(\cI\cS_X),\cD_{\cZ}(\cZ(\cI\cS_X)))$. Hence by Theorem \ref{MV}(a), there exists an isomorphism in $\Hom(\cI\cS_X,\cD (\cI\cS_X))$. Therefore, $\cI\cS_X$ is self-dual in $Perv(X)$.

\end{proof}

\br By our calculation in the proof of Proposition \ref{13}, we obtain the following stalk calculation for the intersection-space complex associated to a hypersurface $X$ with an isolated singular point $x$ satisfying the assumption \ref{ass}:
\be
\cH^r(\cI\cS_X)_x=\begin{cases} \bQ \ \ \ \ \ \ \ \ \ \ \ , \  {\rm for} \ r=-n, \\ H^n(L;\bQ) \ , \ {\rm for} \ r=0  \\ 0 \ \ \ \ \ \ \ \ \ \ \ \ , \  {\rm for} \ r \neq -n,0. \end{cases}
\ee
Moreover, at a smooth point $y \in X^{\circ}$, we get by the splitting of the nearby cycles:
\be
\cH^r(\cI\cS_X)_y=\cH^r(\psi_{\pi}(\bQ_{\wti{X}}[n]))_y=\begin{cases} \bQ  \ \ , \  {\rm for} \ r=-n, \\  0  \ \ \ , \  {\rm for} \ r \neq -n. \end{cases}
\ee

\er

\br Note that one could attempt to {\it define} an intersection-space complex as the unique perverse isomorphism class corresponding to the zig-zag  given by the triple $(\bQ_{X^{\circ}}[n], H^n(F,L;\bQ)/A, H^n(F;\bQ)/B)$ together with the exact sequence (\ref{zzis}). However, the resulting perverse sheaf, while being self-dual, has no clear relation to Hodge theory.
\er


\begin{thebibliography}{DMS}

\bibitem{Ba}   Banagl, M., {\it Intersection spaces, spatial homology truncation, and string theory.} Lecture Notes in Mathematics, 1997. Springer-Verlag, Berlin, 2010. xvi+217 pp.

\bibitem{BaTISS} Banagl, M.,
 {\it Topological invariants of stratified spaces}, Springer
  Monographs in Mathematics, Springer-Verlag Berlin Heidelberg, 2007.

\bibitem{BM}  Banagl M., Maxim. L., {\it Deformation of Singularities and the Homology of Intersection Spaces}. arXiv:1101.4883, Journal of Topology and Analysis (2012).

\bibitem{BBD}  Beilinson, A.A.,  Bernstein, J., Deligne, P., {\it Faisceaux pervers}. In {\it Analysis and topology on singular spaces, I (Luminy, 1981),} 5�171, Ast\'erisque {\bf 100}, Soc. Math. France, Paris, 1982. 


\bibitem{DM}  de Cataldo, M.A., Migliorini,  L.,  {\it The decomposition theorem, perverse sheaves and the topology of algebraic maps}, Bull. Amer. Math. Soc. (N.S.) {\bf 46} (2009), no. 4, 535--633.

\bibitem{DMSS} Davison, B., Maulik, D., Sch\"urmann, J., Szedr\"oi, B., \emph{Purity for graded potentials and quantum cluster positivity}, preprint 2012.

\bibitem{Di}  Dimca, A., {\it Sheaves in topology.} Universitext. Springer-Verlag, Berlin, 2004. xvi+236 pp.

\bibitem{GM1} Goresky, M.; MacPherson, R., \emph{Intersection homology theory}, Topology  {\bf 19} (1980), no. 2, 135--162.

\bibitem{GM2} Goresky, M.; MacPherson, R., \emph{Intersection homology. II.}, Invent. Math.  {\bf 72} (1983), no. 1, 77--129.

\bibitem{GM3} Goresky, M., MacPherson, R., \emph{On the topology of complex algebraic maps}. Algebraic geometry (La Rabida, 1981), 119--129, Lecture Notes in Math.  {\bf 961}, Springer, Berlin, 1982.

\bibitem{GH1} Green, P.S., H\"ubsch, T., \emph{Possible phase transitions among Calabi-Yau compactifications}, Phys. Rev. Lett. {\bf 61} (1988), 1163--1166. 

\bibitem{GH2} Green, P.S., H\"ubsch, T., \emph{Connecting moduli spaces of Calabi-Yau threefolds}, Comm. Math. Phys. {\bf 119} (1988), 431--441. 

\bibitem{huebschworkingdef}
T.~H{\"u}bsch, \emph{On a stringy singular cohomology}, Mod. Phys. Lett. A
  \textbf{12} (1997), 521 -- 533.

\bibitem{KS} Kashiwara, M, Shapira, P., \emph{Sheaves on manifolds}, Springer-Verlag, Berlin, 1994.

\bibitem{MV} MacPherson, R., Vilonen, K., \emph{Elementary construction of perverse sheaves}, Inv. Math. {\bf 84} (1986), 403-435.

\bibitem{Mi} Milnor, J., \emph{Singular points of complex hypersurfaces}, Annals of Mathematics Studies, No. {\bf 61}, Princeton University Press, Princeton, N.J., 1968.


\bibitem{Ma}  Massey, D., {\it Natural Commuting of Vanishing Cycles and the Verdier Dual}. arXiv:0908.2799. 

\bibitem{Ma2} Massey, D., {\it Intersection cohomology, monodromy and the Milnor fiber.} Internat. J. Math.  {\bf 20} (2009), no. 4, 491--507.

\bibitem{Mor} Morrison, D., \emph{Through the looking glass}, Mirror symmetry, III (Montreal, PQ, 1995), 263--277, AMS/IP Stud. Adv. Math.  {\bf 10}, Amer. Math. Soc., Providence, RI, 1999.

\bibitem{R} Rahman, A., \emph{A perverse sheaf approach toward a cohomology theory for string theory}, Adv. Theor. Math. Phys. {\bf 13} (2009), 667-693.

\bibitem{Sa} Saito, M., {\it Modules de Hodge polarisables}, Publ. Res. Inst. Math. Sci.  {\bf 24} (1988), no. 6, 849--995.

\bibitem{Sa1}
Saito, M., {\it Duality for vanishing cycle functors}, Publ. Res. Inst. Math. Sci.   {\bf 25} (1989), 889--921.


\bibitem{Sa2}
Saito, M., {\it Mixed Hodge modules}, Publ. Res. Inst. Math. Sci.   {\bf 26} (1990), 221--333.

\bibitem{Sa3}
Saito, M., {\it  Decomposition theorem for proper K\"ahler morphisms}, Tohoku Math. J. (2) {\bf 42} (1990), no. 2, 127--147.

\bibitem{S} Siersma, D., \emph{The vanishing topology of non isolated singularities}. New developments in singularity theory (Cambridge, 2000), 447--472, 
NATO Sci. Ser. II Math. Phys. Chem., 21, Kluwer Acad. Publ., Dordrecht, 2001.

\end{thebibliography}
\end{document}